\documentclass[10pt]{article}
\usepackage{amsmath}
\usepackage{amsfonts}
\usepackage{amssymb}
\usepackage{multirow}
\usepackage{mathrsfs}
\usepackage[pagewise]{lineno}
\usepackage{xcolor}
\usepackage{tikz}
\usetikzlibrary{arrows,shapes,chains}

\usepackage{graphicx}
\usepackage{listings}
\renewcommand{\Box}{\framebox{\rule{0.3em}{0.0em}}}

\newtheorem{thm}{Theorem}[section]
\newtheorem{lema}{Lemma}[section]
\newtheorem{prop}{Proposition}[section]
\newtheorem{ex}{Example}[section]

\newtheorem{defi}{Definition}[section]

\newcommand{\z}{{ z}}

\newcommand{\bgeqn}{\begin{eqnarray}}
\newcommand{\edeqn}{\end{eqnarray}}
\newcommand{\bgeq}{\begin{eqnarray*}}
\newcommand{\edeq}{\end{eqnarray*}}
\newcommand{\bec}{\begin{center}}
\newcommand{\enc}{\end{center}}

\newcommand{\F}{{\cal F}}

\newcommand{\I}{{\cal I}}

\newcommand{\be}{\begin{equation}}
\newcommand{\ee}{\end{equation}}

\def\dist{\mathop{\rm dist}}

\def\gph {{\rm gph}}

\newtheorem{remark}{Remark}[section]
\renewcommand{\Box}{\hfill \rule{2.3mm}{2.3mm}}

\makeatletter
\def\@biblabel#1{#1.}
\makeatother

\makeatletter

\newenvironment{proof}{\noindent{\bf Proof. }}{\hfill $\Box$\medskip}
\renewcommand{\Box}{\framebox{\rule{0.1em}{0.0em}}}

\begin{document}
\renewcommand{\thefootnote}{}
\begin{center}
{\Large  Optimality Conditions and Exact Penalty for Mathematical Programs with
Switching Constraints\footnote{The first author's work was supported in part by NSFC Grant \#11801152, \#12071133, \#11671122. The second author's work was supported by NSERC.}}

{\large Yan-Chao
Liang\footnote{Yan-Chao
Liang, College of Mathematics and Information Science, Henan Normal University, Xinxiang 453007, China. Email: {\tt liangyanchao83@163.com.}}
\ and\  Jane J. Ye\footnote{Corresponding author: Jane J. Ye, Department of Mathematics and Statistics, University of Victoria, Victoria, BC, V8W 2Y2,
Canada. Email: {\tt janeye@uvic.ca.}}}

\bigskip
\end{center}

{\vspace{13pt} \noindent{\bf Abstract.} In this paper, we give {an overview} on optimality conditions and exact penalization for the mathematical program with switching constraints (MPSC).
MPSC is a new class of optimization problems with important  applications.
It is well known that if MPSC is treated as a standard nonlinear program,  some of the usual constraint qualifications may fail. {To} deal with this issue, one could  reformulate it as a mathematical program with disjunctive constraints (MPDC). In this paper, we first survey  recent results on constraint qualifications and optimality conditions for MPDC,  then apply them to MPSC.
Moreover, we provide two types of sufficient conditions for {the local error bound}   and   exact penalty results for MPSC. One comes from the directional quasi-normality for MPDC,  and the other is obtained via the local decomposition approach.

}


\baselineskip 18pt

\section{Introduction}

The mathematical
program with switching constraints (MPSC) defines a class of optimization problems  in which some of the equality constraint functions  are  products of two  functions. The terminology ``switching constraint'' comes from the fact that if the product of two {constraint} functions is equal to zero, then at least one of them must be equal to zero. {MPSC can be used to model the discretized version of the optimal control problem with switching structure (see e.g. \cite{Clason,kanzow-mehlitz-steck,Mehlitz MP} and the references therein),  or to reformulate the so-called  mathematical programs  with either-or-constraints (see \cite[Section 7]{Mehlitz MP}). MPSC has many interesting applications, e.g., optimal control with switching structures have been used to model certain real-world applications  \cite{Gugat,Hante,Seidman} and the mathematical program with {either-or-constraints} was used to study some special instances of portfolio optimization  \cite{kanzow-mehlitz-steck}.}

It is {well known} that if the mathematical program with equilibrium constraints (MPEC)  \cite{LuoPangRalph,outrata}, or the mathematical program with vanishing constraints (MPVC) \cite{Achitziger,Hoheisel},
are treated as nonlinear programs,  then there are issues involving the usual constraint qualifications such as the Mangasarian-Fromovitz Constraint Qualification (MFCQ) and/or the linear independence constraint qualification (LICQ) for MPEC and MPVC. It is not surprising that this issue also exists for MPSC.
Indeed, Mehlitz  \cite[Lemma 4.1]{Mehlitz MP} showed that if an MPSC is treated as a nonlinear program,  then MFCQ fails at any feasible point $z^*$
for which there is a pair of switching functions with value equal to zero.
Consequently, he introduced the concepts of weak, Mordukhovich (M-), and strong (S-) stationarity for MPSC and presented some associated constraint qualifications.
Kanzow et al. \cite{kanzow-mehlitz-steck} adopted several relaxation methods  from the numerical treatment of MPEC to MPSC. Li and Guo extended some weak and verifiable constraint qualifications for nonlinear programs  to MPSC  in \cite{Li-Guo}. In the work of \cite{Mehlitz MP,Li-Guo}, the error bound property was not studied and that is one of the main {focuses} of this paper.

{The mathematical {program} with disjunctive constraints (also called the disjunctive program) is a  {type of set-constrained  optimization problem where the set is the union of finitely many polyhedral convex sets.}
Programs such as  MPEC, MPVC and MPSC can be reformulated as  disjunctive programs.  The classical concepts of  optimality for  disjunctive programs such as  S-stationary condition based on the regular normal cone and  M-stationary condition based on the limiting/Mordukhovich normal cone for disjunctive programs were introduced by Flegel et al. \cite{Flegel etal 2007}. Although M-stationary condition holds for a local minimizer under very weak constraint {qualifications} such as the generalized Guignard constraint qualification (GGCQ), it may be weak for some problems  and it does not exclude feasible descent directions.
Based on concepts of metric subregularity and some new developments in variational analysis, for disjunctive programs,
Gfrerer \cite{Gfrerer SIAM Optimal 2014} introduced various new concepts of constraint qualifications and stationarity concepts including the strong M-stationarity and  the extended M-stationarity which are stronger than  M-stationarity.
Moreover, a directional version of LICQ and directional first and second order optimality conditions are given in \cite{Gfrerer SIAM Optimal 2014}.  Another direction of sharpening optimality conditions and weakening constraint qualifications is to consider  directional optimality conditions and constraint qualifications. Bai et al. \cite{Bai-Ye-Zhang2019} introduced the directional  quasi/pseudo-normality as sufficent conditions for the metric subregularity {which} are weaker than both the classical quasi/pseudo-normality and the first order sufficient condition for metric subregularity. {Benko et al. \cite{Benko etal2019} generalized the notions of directional pseudo- and quasi-normality to obtain more sufficient conditions for metric subregularity. In particular,  they have shown that for the disjunctive program, the (directional) pseudo-normality can always take the simplified form while  for a special class of the disjunctive program  called the ortho-disjunctive program (which includes MPSC), the (directional)  quasi-normality can also take the simplified form. Mehlitz \cite{Mehlitz2020optimization} introduced an alternative concept of  LICQ  and {obtained} first and second
order optimality conditions for disjunctive programs.}
{Recall that  M-stationary condition does not preclude the existence of feasible descent directions. To deal with this issue,  recently }Benko and Gfrerer \cite{Benko-Gfrerer2017} introduced  the so-called $\mathcal{Q}$-stationarity and $\mathcal{Q}_M$-stationarity {where $\mathcal{Q}_M$-stationarity is stronger than  M-stationarity}. 
A further extension of $\mathcal{Q}$-stationarity and $\mathcal{Q}_M$-stationarity are presented in  Benko and Gfrerer \cite{Benko-Gfrerer2018}.  To deal with the difficulty of calculating the limiting normal cone to the feasible region,  Gfrerer \cite{Gfrerer2019} introduced a new concept of stationary condition for a set-constrained optimization problem called the linearized M-stationary condition.
{Recently, sequential optimality conditions and constraint qualifications and their applications in numerical algorithms became a popular topic.  A suitable theory has been developed in the context of  MPEC in \cite{Andreani etal2019,Ramos2021}. Mehlitz \cite{Mehlitz2020} has generalized the underlying theories to an very general optimization problem which includes MPDC as a special case.}

In this paper, we will survey the aforementioned  results about new stationarity concepts  and sufficient conditions for metric subregularity for  disjunctive programs. We then apply these results to obtain various optimality conditions and local error bound results for MPSC. Moreover, we propose to use the local decomposition approach to study  sufficient conditions for the error bound property by the corresponding constraint qualifications for each branch as a standard nonlinear program (NLP).

%
%

The remainder of this paper is organized as follows. {In} Section 2, we  review some constraint qualifications from nonlinear programs,  and existing constraint qualifications and optimality conditions for MPSC.  In Section 3, we summarize the results that we need for disjunctive programs. In Section 4, we apply the results from Section 3 to MPSC. In Section 5, we derive the local error bound and exact penalty results for MPSC. 
In Section 6, we conclude our discussion and provide relationships among various constraint qualifications, error bound properties and stationary conditions.

Throughout the paper, for a differentiable mapping $c : \mathbb{R}^n \rightarrow \mathbb{R}^m$ and a vector $z\in \mathbb{R}^n$, we denote by $\nabla c(z)$ {the Jacobian} of $c$ at $z$.
For a  differentiable function $f : \mathbb{R}^n \rightarrow \mathbb{R}$, we denote by $\nabla f(z)$ its gradient vector and $\nabla^2 f(z)$ its Hessian matrix at $z$ provided that it is twice differentiable.
For a set ${\cal C}$, we denote by ${\cal C}^\circ:=\{x\ |\ x^Ty\leq0,\forall y\in{\cal C}\}$ its polar cone,  
and by ${\rm dist}_{\cal C}(x)$ the distance between  $x$ and ${\cal C}$. Unless otherwise specified,
$\|\cdot\|$ denotes an arbitrary norm in $\mathbb{R}^n$.

\section{Review of constraint qualifications and optimality conditions}
In this section, we first  recall some constraint qualifications for NLP. Then we review some existing constraint qualifications and optimality conditions for MPSC.
The reader is referred to  \cite{Mehlitz MP,Li-Guo} for those constraint qualifications for MPSC that are not reviewed here.

\subsection{Constraint qualifications for NLP}
Consider the standard nonlinear program
\begin{eqnarray}
\min f(z)  &&
{\rm s.t.}\  g(z)\leq 0,
\ h (z)=0,\ \label{standard NLP}
\end{eqnarray}
where ${f}:\mathbb{R}^n\to \mathbb{R}$, ${g}:\mathbb{R}^n\to\mathbb{R}^p$, ${h}:\mathbb{R}^n\to\mathbb{R}^q$ are continuously differentiable.
Denote by $\bar{\I}_{{g}}:={\I}_{{g}}(\bar z)=\{i\in \{1,\cdots,p\}|{g}_i(\bar{z})=0\}$ the index set of active inequality constraints at $\bar z$.
We recall some constraint qualifications for problem (\ref{standard NLP}) that we will refer to in this paper.

\begin{defi}\rm
Let $\bar{z}\in\mathbb{R}^n$ be a feasible point of problem (\ref{standard NLP}). We say that $\bar{z}$ satisfies
\begin{itemize}
\item[1.] {\em linear independence constraint qualification} (LICQ), if the family of gradients $\{\nabla {g}_i(\bar{z})\}_{i\in \bar{\I}_{{g}}}\cup\{\nabla {h}_i(\bar{z})\}_{i=1}^q$ is linearly independent;

\item[2.] {\em Mangasarian-Fromovitz constraint qualification} (MFCQ) \cite{MFCQ1967}, or equivalently {\em positive-linearly independent constraint qualification} (PLICQ) if the family of gradients $\{\nabla {g}_i(\bar{z})\}_{i\in \bar{\I}_{{g}}}\cup\{\nabla {h}_i(\bar{z})\}_{i=1}^q$ is positive-linearly independent, i.e. the family of gradients $\{\nabla {g}_i(\bar{z})\}_{i\in \bar{\I}_{{g}}}\cup\{\nabla {h}_i(\bar{z})\}_{i=1}^q$ is linearly independent with non-negative scalars associated to the gradients of the active inequality constraints;

\item[3.] {\em constant rank constraint qualification} (CRCQ) \cite{Janin1984}, if there exists a neighborhood $N(\bar{z})$ of $\bar z$ such
that for every $I\subseteq \bar{\I}_{{g}}$ and every $J \subseteq \{1,\cdots,q\}$, the family of gradients
$\{\nabla {g}_i (z)\}_{i\in I} \cup \{\nabla {h}_i (z)\}_{i\in J} $ has the same rank for every $z \in N(\bar{z})$;
\item[4.] {\em relaxed constant rank constraint qualification} (RCRCQ) \cite{minchenko-stakhovski2011}, if there exists a neighborhood $N(\bar{z})$ of $\bar z$ such
that for every $I\subseteq \bar{\I}_{{g}}$ , the family of gradients
$\{\nabla {g}_i (z)\}_{i\in I}\cup\{\nabla {h}_i (z)\}_{i=1}^q$ has the same rank for every $z \in N(\bar{z})$;
\item[5.] {\em constant positive linear dependence constraint
qualification} (CPLD) \cite{Qi-Wei2000}, if there exists a neighborhood $N(\bar{z})$ of $\bar z$ such
that for every $I\subseteq \bar{\I}_{{g}}$ and every $J \subseteq \{1,\cdots,q\}$, whenever the family of gradients $\{\nabla {g}_i(\bar z)\}_{i\in I} \cup \{\nabla {h}_i (\bar z)\}_{i\in J} $ is positive-linearly dependent, then
$\{\nabla {g}_i (z)\}_{i\in I} \cup\{\nabla {h}_i (z)\}_{i\in J}$ is linearly dependent for every $z \in N(\bar{z})$;
\item[6.] {\em relaxed constant positive linear dependence constraint
qualification} (RCPLD) \cite{Andreani etal 2012MP}, if there exists a neighborhood $N(\bar{z})$ of $\bar{z}$ such
that (i) $\{\nabla {h}_i (z)\}_{i=1}^q$ has the same rank for every $z \in N(\bar{z})$;
(ii) For every $I\subseteq \bar{\I}_{{g}}$, if the family of gradients $\{\nabla {g}_i(\bar z)\}_{i\in I} \cup \{\nabla {h}_i (\bar z)\}_{i\in J} $ is positive-linearly dependent,  where $J\subseteq \{1,\cdots,q\}$ is such that $\{\nabla {h}_i (\bar{z})\}_{i\in J}$ is {a} basis for span $\{\nabla {h}_i (\bar{z})\}_{i=1}^q$, then
$\{\nabla {g}_i (z)\}_{i\in I} \cup \{\nabla {h}_i (z)\}_{i\in J}$ is linearly dependent for every $z \in N(\bar{z})$;
\item[7.] {\em constant rank of subspace component }(CRSC)  \cite{Andreani etal 2012}, 
 if there exists a neighborhood $N(\bar{z})$ of $\bar{z}$ such
that the rank of $\{\nabla {g}_i(z)\}_{i\in \I^-}\cup\{\nabla {h}_i(z)\}_{i=1}^q$ remains constant for $z\in N(\bar{z})$, where $$\I^-=:\left \{l \in  \bar{\I}_{{g}} {\left|-\nabla g_l(\bar z)\in \left \{\sum_{i=1}^q \lambda_i \nabla h_i(\bar z)+\sum_{i\in  \bar{\I}_{{g}}\setminus \{l\}}  \mu_i \nabla g_i(\bar z)|\mu_i \geq 0,  i\in \bar{\I}_{{g}}\right\}\right.} \right \}.$$
\end{itemize}
\end{defi}
\begin{remark}  Let $L(\bar{z}):=\{d|\nabla g_i(\bar{z})d\leq0,\ i\in \bar \I_g,\nabla h_i(\bar{z})d=0,i=1,\cdots,q\}$ be the linearization cone of problem (\ref{standard NLP})  at $\bar{z}$.
Kruger al.  \cite{Kruger etal2014} pointed out that since the polar of the linearization cone is equal to
$$ L(\bar{z})^\circ= \left \{{\left.\sum_{i=1}^q \lambda_i \nabla h_i(\bar z)+\sum_{i\in  \bar{\I}_{{g}}}  \mu_i \nabla g_i(\bar z)\right|}\mu_i \geq 0,  i\in \bar{\I}_{{g}}\right \},$$
by the definition of the linearization cone, the index set $\I^-$ can be equivalently written as
$${\I^-=\{l\in \bar{\I}_g|\nabla g_l(\bar{z})^{{T}}d=0,\forall d\in L(\bar{z})\}.}$$
Hence in  \cite{Kruger etal2014},  CRSC is also called  {\em relaxed MFCQ}.
\end{remark}
\begin{defi} \rm(see e.g.
\cite{Solodov2011})
Let ${\F}_{NLP}$ be the feasible region of problem (\ref{standard NLP}).
We say that an {\em error bound} holds in a neighborhood $N(\bar{z})$ of a feasible point $\bar{z}\in{\F}_{NLP}$ if there exists $\alpha>0$ such that for every $z\in N(\bar{z})$
$${\rm dist}_{{\F}_{NLP}}(z)\leq\alpha\left(\sum_{i=1}^p\max\{g_i(z),0\}+\sum_{i=1}^q|h_i(z)|\right).$$
\end{defi}

{It is easy to see that the local error bound condition holds at $\bar{z}$ for NLP if and only if the feasibility mapping $z\rightrightarrows (g(z),h(z))-\mathbb{R}^p_-\times\{0\}^q$ is {\em metrically subregular} (see Definition \ref{definition of directional metric subregularity}) at $(\bar{z},0)$. }

Andreani al. \cite[Theorem 5.5]{Andreani etal 2012}  showed that  CRSC implies the existence of local error
bounds under the second-order differentiability of functions $g, h$. This assumption  was removed by Guo et al.   \cite{Guo-Zhang-Lin2014}.
Finally, we summarize
relations among constraint qualifications for NLP discussed in this subsection in Figure 1.
\begin{figure}%
\centering
		\scriptsize
		 \tikzstyle{format}=[rectangle,draw,thin,fill=white]
		 \tikzstyle{test}=[diamond,aspect=2,draw,thin]
		\tikzstyle{point}=[coordinate,on grid,]
		\begin{tikzpicture}
        \node[format](LICQ){LICQ};
        \node[format,right of=LICQ,node distance=15mm](MFCQ){MFCQ};
        \node[format,right of=MFCQ,node distance=15mm](CPLD){CPLD};
		\node[format,right of=CPLD,node distance=15mm](RCPLD){RCPLD};
       \node[format,right of=RCPLD,node distance=15mm](CRSC){CRSC};
       \node[format, below of=MFCQ,node distance=7mm](CRCQ){CRCQ};
       \node[format,right of=CRCQ,node distance=15mm](RCRCQ){RCRCQ};
       \node[format,right of=CRSC,node distance=27mm](error bound){error bound/metric subregular};
\draw[->](LICQ)--(MFCQ);
\draw[->](MFCQ)--(CPLD);
\draw[->](CPLD)--(RCPLD);
\draw[->](RCPLD)--(CRSC);
\draw[->](LICQ)--(CRCQ);
\draw[->](CRCQ)--(RCRCQ);
\draw[->](CRCQ)--(CPLD);
\draw[->](RCRCQ)--(RCPLD);
\draw[->](CRSC)--(error bound);
\end{tikzpicture}

\centering{Fig.1  Relation among constraint qualifications for NLPs}
\end{figure}

\subsection{Constraint qualifications and optimality conditions for MPSC}\label{s2.2}
In this paper, we consider the following MPSC:
\begin{eqnarray}
\min   & & f(z) \nonumber\\
\mbox{s.t.} & & g(z)\le 0,  \ h(z)=0, \quad G_i(z) H_i(z)=0,\ i=1,\cdots,m \label{MPSC}
\end{eqnarray}
where $f:\mathbb{R}^n\rightarrow \mathbb{R},~g:\mathbb{R}^n\rightarrow \mathbb{R}^p,~h:\mathbb{R}^n\rightarrow \mathbb{R}^q,~G_1,\cdots,G_m:\mathbb{R}^n\rightarrow \mathbb{R},~H_1\cdots,H_m:\mathbb{R}^n\rightarrow \mathbb{R}$. We assume that unless otherwise specified, all defining functions are continuously  differentiable.
Let ${\cal F}$ denote the feasible region of (\ref{MPSC}). For
a feasible point $z^*\in\F$ , we define some useful index sets as follows:
\begin{eqnarray*}
&&{\cal I}_g^*:={\cal I}_g(z^*)=\{i\in\{1,\cdots,p\}\ |\ g_i(z^*)=0\}, \\
&&{\cal I}_G^*:={\cal I}_G(z^*)=\{i\in\{i,\cdots,m\}\ |\ G_i(z^*)=0,\ H_i(z^*)\neq 0\},\\
&&{\cal I}_H^*:={\cal I}_H(z^*)=\{i\in\{i,\cdots,m\}\ |\ G_i(z^*)\neq 0,\ H_i(z^*)=0\},\\
&&{\cal I}_{GH}^*:={\cal I}_{GH}(z^*)=\{i\in\{i,\cdots,m\}\ |\ G_i(z^*)=H_i(z^*)=0\}.
\end{eqnarray*}
Since by \cite[Lemma 4.1]{Mehlitz MP},  MPSC never {satisfies} MFCQ at a feasible point $z^*$ with  $ {\cal I}_{GH}^*\not =\emptyset$, Mehlitz \cite{Mehlitz MP}  defined and studied the following alternative stationarity concepts.
\begin{defi}\rm\cite{Mehlitz MP} \label{WMS}
We say that $z^*\in\F$ is a {\em weakly  stationary} (W-stationary) point of MPSC (\ref{MPSC}) if there exist multipliers $(\lambda^g,\lambda^h,\lambda^G,\lambda^H)$ such that
\begin{eqnarray}
&&\nabla f(z^*)+\sum_{i\in\I_g^*}\lambda^g_i\nabla g_i(z^*)+\sum_{i=1}^q\lambda^h_i\nabla h_i(z^*)+\sum_{i=1}^m(\lambda^G_i\nabla G_i(z^*)+\lambda^H_i\nabla H_i(z^*))=0,\label{W-stationary1}\\
&&\lambda^g_i\geq 0,\ i\in\I_g^*,
\lambda^G_i=0,\ i\in\I_H^*,
\lambda^H_i=0,\ i\in\I_G^*.\label{W-stationary4}
\end{eqnarray}

We say that $z^*\in\F$ is a {\em Mordukhovich stationary} (M-stationary) point of MPSC (\ref{MPSC}) if there exist multipliers $(\lambda^g,\lambda^h,\lambda^G,\lambda^H)$ such that (\ref{W-stationary1})--(\ref{W-stationary4}) hold and
$\lambda^G_i\lambda^H_i=0,\ i\in\I_{GH}^*.$ Moreover, we call $(\lambda^g,\lambda^h,\lambda^G,\lambda^H)$ an M-multiplier.

We say that $z^*\in\F$ is a {\em strongly stationary} (S-stationary) point of MPSC (\ref{MPSC}) if there exist multipliers $(\lambda^g,\lambda^h,\lambda^G,\lambda^H)$ such that (\ref{W-stationary1})--(\ref{W-stationary4}) hold and
$\lambda^G_i=\lambda^H_i=0,\ i\in\I_{GH}^*.$ Moreover, we call $(\lambda^g,\lambda^h,\lambda^G,\lambda^H)$ an S-multiplier.
\end{defi}

Consider the associated  tightened nonlinear problem at $z^*\in\F$:
\begin{eqnarray*}
({\rm TNLP})\ \min &&f(z)\\
{{\rm s.t.}}&&g(z)\leq 0,
h(z)=0, G_i(z)=0,\ i\in\I_G^*\cup\I_{GH}^*,H_i(z)=0,\ i\in\I_H^*\cup\I_{GH}^*.
\end{eqnarray*}

\begin{defi}\rm\cite{Mehlitz MP}\label{definition of MPSC LICQ and MFCQ}
Let $z^*$ be a feasible point of MPSC (\ref{MPSC}). We say that $z^*$ satisfies
MPSC-LICQ/-MFCQ, if LICQ/MFCQ holds for (TNLP) at $z^*$.
%
%
\end{defi}
\begin{defi}\rm\label{New CQ definition-1}
Let $z^*$ be a feasible point of MPSC (\ref{MPSC}). We say that $z^*$  satisfies
 MPSC-CRCQ/-CPLD, if CRCQ/CPLD holds for (TNLP) at $z^*$.
\end{defi}

\begin{remark}
The  MPSC-CRCQ/-CPLD defined in Definition \ref{New CQ definition-1} coincides with those defined in \cite[Definition 4.2]{Li-Guo}. The advantage of defining these constraint qualifications as the corresponding ones for the tightened nonlinear program (TNLP) is that we can immediately conclude from the definitions of CRCQ and CPLD for nonlinear programs that MPSC-CRCQ implies MPSC-CPLD without proof as in (i) of the proof for
 \cite[Theorem 4.2]{Li-Guo}.
\end{remark}

\begin{defi}[MPSC-RCPLD]\rm \cite{Li-Guo}
Let $z^*$ be a feasible point of MPSC (\ref{MPSC}). We say that $z^*$  satisfies
 MPSC-RCPLD if there exists a neighborhood $N(z^*)$ of $z^*$ such that
 \begin{itemize}
 \item[(i)] The vectors $\{\nabla h_i(z)\}_{i=1}^q\cup \{ \nabla G_i(z)\}_{i\in \mathcal{I}_G^*} \cup  \{ \nabla H_i(z)\}_{i\in \mathcal{I}_H^*}$ have the same rank for all $z$ in $N(z^*)$;
 \item[(ii)]
 Let $I_1\subseteq\{1,2,\cdots,q\},I_2\subseteq\I_G^*,I_3\subseteq\I_H^*$ be index sets such that the set of vectors
 $\{\nabla h_i(z^*)\}_{i\in I_1}\cup \{ \nabla G_i(z^*)\}_{i\in I_2} \cup  \{ \nabla H_i(z^*)\}_{i\in I_3}$ is a basis for span {$(\{\nabla h_i(z^*)\}_{i=1}^q, \{ \nabla G_i(z^*)\}_{i\in \mathcal{I}_G^*} ,  \{ \nabla H_i(z^*)\}_{i\in \mathcal{I}_H^*})$}.  For each $I_4\subseteq\I_g^*,$ $I_5,I_6\subseteq\I_{GH}^*$, if there exist $\{\lambda^g,\lambda^h,\lambda^G,\lambda^H\}$ not all zero, with $\lambda^g_i\geq0$ for each $i\in I_4$ and $\lambda^G_i\lambda^H_i=0$ for each $i\in\I_{GH}^*$, such that
\begin{eqnarray*}
\sum_{i\in I_4}\lambda^g_i\nabla g_i(z^*)+\sum_{i\in I_1}\lambda^h_i\nabla h_i(z^*)+\sum_{i\in I_2\cup I_5}\lambda^G_i\nabla G_i(z^*)+\sum_{i\in I_3\cup I_6}\lambda^H_i\nabla H_i(z^*)=0,
\end{eqnarray*}
then for any $z\in N(z^*)$, the set of vectors
$$\left \{\nabla g_i(z) \right \}_{i\in I_4} \cup \{\nabla h_i(z)\}_{i\in I_1}\cup\{\nabla G_i(z)\}_{i\in I_2\cup I_5}\cup
\{\nabla H_i(z)\}_{i\in I_3\cup I_6}$$ {is} linearly dependent.
\end{itemize}
\end{defi}

We now gather  constraint  conditions and necessary optimality conditions from \cite{Mehlitz MP,Li-Guo} in the following theorem.
One can find the definition of  MPSC No Nonzero Abnormal Multiplier Constraint Qualification (MPSC-NNAMCQ)  and MPSC quasi/pseudo-normality from the comments after Definitions \ref{Defi4.3} and \ref{Defi4.4}  respectively.
\begin{thm}{\rm \cite{Mehlitz MP,Li-Guo}}
Let $z^*$ be feasible for problem (\ref{MPSC}). If  MPSC-LICQ is fulfilled at $z^*$, then $z^*$ is  S-stationary. If {MPSC-MFCQ/-CPLD/-CRCQ/-RCPLD/-NNAMCQ/-quasi/-pseudo-normality}  is fulfilled at $z^*$, then $z^*$ is M-stationary.
\end{thm}

\section{Optimality conditions for mathematical programs with disjunctive constraints}\label{sec3}
In this section we review some  optimality conditions for the mathematical programs with disjunctive constraints (MPDC):
\begin{eqnarray}\label{disjunctive constraints}
\min &&f(z)
\quad {\rm s.t.} \ P(z)\in\Lambda,
\end{eqnarray}
where $f:\mathbb{R}^n \rightarrow \mathbb{R}$, $ P:\mathbb{R}^n \rightarrow \mathbb{R}^{m}$ are continuously differentiable,  and $\Lambda \subseteq \mathbb{R}^{m}$ is the union of finitely many convex polyhedral sets. We denote the feasible region by ${\cal F}_{D}:=\left \{z\in \mathbb{R}^{n}| P(z)\in \Lambda \right \}$ and the linearization cone by $$L_{{\cal F}_D}^{\rm lin}(z^*):=\{ d\in \mathbb{R}^{n}|\nabla P(z^*) d \in T_\Lambda (P(z^*))\}.$$}

To study the mathematical program with disjunctive constraints (\ref{disjunctive constraints}), we need to study various tangent cones and normal cones to set $\Lambda$. First we recall definitions of tangent cones and normal cones.
Suppose that $A\subseteq \mathbb{R}^m$ is closed and  $x^*\in A$.
The tangent/Bouligand  cone, the Fr\'{e}chet/regular and the limiting/basic/Mordukhovich normal cone to $A$ at $x^*$ are defined by
\begin{eqnarray*}
T_{A}(x^*)
&:=&\left \{d\in \mathbb{R}^m\ |\ \exists\ t_k\downarrow 0,\  d^k\rightarrow d\ \text{such\ that}\ x^*+t_k d^k\in A\right \},\\
\widehat{N}_A(x^*)&:=&\left \{\zeta\in\mathbb{R}^m| \langle \zeta, x-x^*\rangle \leq o(\|x-x^*\|) \quad \forall x \in A \right \},\\
N_A(x^*)&:=& \left \{\zeta\in\mathbb{R}^m\  {\left|\ \exists\ \{x^k\}\subseteq A,\ \exists \zeta^k\ {\rm such\ that}\ x^k\to x^*,\ \zeta^k\to \zeta,\ \zeta^k\in\widehat{N}_A(x^k) \right.}\right \},
\end{eqnarray*}
respectively,  see e.g. \cite{Rockafellar}. When $A$ is convex, all the  normal cones above are equal and they coincide with the normal cone in the sense of convex analysis.

Using various  normal cones, some stationary conditions were introduced; cf \cite[Definition 1]{Flegel etal 2007}, \cite[Definition 3]{Benko-Gfrerer2017}.
\begin{defi}\rm Let $z^*\in {\cal F}_D$.
\begin{itemize}
\item[(a)] We say that $z^*$ is B-stationary (Bouligand stationary) if
$0\in  \nabla f(z^*)+ \widehat{N}_{{\cal F} _D}(z^*).$
\item[(b)] We say that $z^*$ is S-stationary (Strongly stationary) if
$ 0\in  \nabla f(z^*)+  \nabla P(z^*)^T \widehat{N}_\Lambda( P(z^*)).$
\item[(c)] We say that $z^*$ is M-stationary (Mordukhovich stationary) if
$0\in  \nabla f(z^*)+ \nabla P(z^*)^T N_\Lambda (P(z^*)).$
\end{itemize}
\end{defi}
\begin{defi}\rm Let $z^*\in {{\cal F} _D}$. We say that the generalized Guignard constraint qualification (GGCQ) holds at $z^*$ if
\begin{equation}\widehat{N}_{{\cal F} _D}(z^*)=(L_{{\cal F} _D}^{\rm lin}(z^*))^\circ. \label{GGCQ}\end{equation}
\end{defi}

GGCQ is a rather weak constraint qualification. For example,  it holds  if the set-valued map $F(z):=-P(z)+\Lambda$ is metrically subregular at $z^*$ (see Definition \ref{definition of directional metric subregularity}).

The following necessary optimality conditions are well known. Proposition \ref{prop3.1}(a) follows from the well-known fact that any local optimizer has no feasible descent directions and the fact that $(T_{{\cal F} _D}(z^*))^\circ=\widehat{N}_{{\cal F} _D}(z^*)$. Proposition \ref{prop3.1}(b) follows from  (a) and the change of coordinates formula in \cite[Exercise 6.7]{Rockafellar}.
\begin{prop}\label{prop3.1} Let $z^*$ be a local optimal solution of problem (\ref{disjunctive constraints}).  Then
\begin{itemize}
\item[(a)]  $z^*$ is B-stationary.
\item[(b)] If $\nabla P(z^*)$ has full rank $m$, then $z^*$ is S-stationary.
\item[(c)]\cite[Theorem 7]{Flegel etal 2007} Suppose GGCQ holds at $z^*$. Then $z^*$ is M-stationary.
\end{itemize}
\end{prop}

Recently some MPDC-tailored versions of LICQ have been introduced in \cite[Definition 3.1]{Mehlitz2020optimization} and
 in \cite[(31)]{GfrererYeZhou} (see also Definition \ref{definition of MPSC LICQ(d)}). These conditions all ensure that a local optimal solution is S-stationary.

It is easy to see that although B-stationary condition does not need any constraint qualification, it is implicit and hence not easy to use.  S-stationary condition is sharper than  M-stationary condition but requires very strong constraint qualification to hold.  M-stationary condition is necessary for optimality under very weak constraint qualification but it can be very weak for certain  problems. {Recently, some stationary conditions weaker than B-stationarity but stronger than M-stationarity have been introduced.} We now review these results.

{For problems in the form (\ref{disjunctive constraints}) but with $\Lambda$ being  an arbitrary  closed set,  the limiting normal cone in M-stationary condition can be hard to compute and the resulting M-stationary condition can be weak. In order to deal with this difficulty, Gfrerer \cite{Gfrerer2019} introduced the so-called linearized M-stationary condition by a repeated linearization procedure. We now apply the linearization procedure to our problem. Suppose that $z^*$ is  B-stationary  for problem (\ref{disjunctive constraints}) with $\Lambda$ being  a closed set and
 GGCQ holds at $z^*$. Then since $-\nabla f(z^*)\in  \widehat{N}_{{\cal F} _D}(z^*) $,  by (\ref{GGCQ})
 the point $d^*=0$ is a global minimizer for the linearized problem
\begin{equation} \label{Linearizedproblem} \min_d  \nabla f(z^*)^Td \mbox{ subject to  } \nabla P(z^*) d \in T_\Lambda (P(z^*)).\end{equation}
If  a constraint qualification holds, then  M-stationary condition  for  the  above linearized problem holds  at $d^*=0$.  In our case since $\Lambda$ is the union of finitely many convex polyhedral sets, the perturbed feasible  map ${{\cal F} _D}(d):= \nabla P(z^*) d - T_\Lambda (P(z^*))$ is metric subregular at $(0,0)$  and hence GGCQ holds automatically.  Then by Proposition \ref{prop3.1}(c), $d^*=0$ is an M-stationary point of the linearized problem which means that
 \begin{equation}\label{LinearM} 0\in \nabla f(\z^*) +\nabla P(z^*)^T  {N}_{T_\Lambda(P(z^*))} (0)
.\end{equation}
The linearization procedure would continue if $T_\Lambda(P(z^*))$ is not the union of finitely many convex polyhedral sets,  until a series of tangent cones to tangent cones to the set $\Lambda$ is the union of finitely many convex polyhedral sets. The resulting optimality condition is called a linearized M-stationary condition.
In general,  the linearized M-stationary condition is sharper  than M-stationary condition. To see this, suppose $T_\Lambda(P(z^*))$ is the union of finitely many convex polyhedral sets and the original set $\Lambda$ is not.  Then the linearized M-stationary condition is (\ref{LinearM}). Since  ${N}_{T_\Lambda(P(z^*))} (0)\subseteq N_\Lambda(P(z^*))$, cf.  \cite[Proposition 6.27]{Rockafellar}, the linearized M-stationary condition is sharper  than M-stationary condition. Moreover,
${N}_{T_\Lambda(P(z^*))} (0)$ would be easier to calculate than the normal cone $N_\Lambda(P(z^*))$ in this case.
But since in our case,} $\Lambda$ is the union of finitely many convex polyhedral sets, ${N}_{T_\Lambda(P(z^*))} (0)=N_\Lambda(P(z^*))$ (see  \cite[p. 59]{Henrion}. Hence the linearized M-stationary condition coincides with  M-stationary condition for the disjunctive program.


{Another approach taken by Benko and Gfrerer in \cite{Benko-Gfrerer2017} to obtain sharper stationary condition than M-stationary condition for problems in the form (\ref{disjunctive constraints})  but with $\Lambda$ being  an arbitrary  closed set is  to give an accurate estimate for  the regular normal cone  to the constraint system.
The idea is as follows. Let $Q_1,Q_2\subseteq T_\Lambda(P(z^*))$ be two closed convex cones. Then
$$\nabla  P(z^*)^{-1}Q_i \subseteq \nabla P(z^*)^{-1}T_\Lambda(P(z^*))=L_{{\cal F} _D}^{\rm lin}(z^*),\quad i=1,2,$$
where
$\nabla P(z^*)^{-1}Q_i:=\{ d| \nabla P(z^*) d\in Q_i\}.$
Therefore, if GGCQ holds at $z^*$, we have
\begin{eqnarray*}
\widehat{N}_{{\cal F} _D}(z^*) &=& (L_{{\cal F} _D}^{\rm lin}(z^*))^\circ \\
&\subseteq & (\nabla  P(z^*)^{-1}Q_1 \cup \nabla  P(z^*)^{-1}Q_2)^\circ\\
&=& (\nabla  P(z^*)^{-1}Q_1)^\circ  \cap ( \nabla  P(z^*)^{-1}Q_2)^\circ \qquad \mbox{ since $Q_1$ and $Q_2$ are convex cones}.
\end{eqnarray*}
Moreover, suppose the following condition holds:
\begin{equation}
(\nabla P(z^*)^{-1}Q_i)^\circ=\nabla P(z^*)^TQ_i^\circ,i=1,2. \label{validity} \end{equation}
Then it follows that
\begin{equation}\widehat{N}_{{\cal F} _D}(z^*) \subseteq  (\nabla  P(z^*)^T Q_1 ^\circ ) \cap ( \nabla  P(z^*)^T Q_2^\circ)= \nabla  P(z^*)^T \left(Q_1 ^\circ \cap ( \ker \nabla  P(z^*)^T+ Q_2^\circ)\right),\label{tight}\end{equation}
where  $\ker\nabla P(z^*)^T:=\{r| P(z^*)^Tr=0\}$ and the equality follows from \cite[Lemma 1]{Benko-Gfrerer2017}. The right hand side of the inclusion (\ref{tight})  gives an upper estimate for $\widehat{N}_{{\cal F} _D}(z^*)$.
In order to have that the above inclusion provides a good estimate for the regular normal cone, it is obvious that we want to choose $Q_1, Q_2$ as large as possible so that the inclusion is tight. Furthermore, since one always has $\nabla P(z^*)^T \widehat N_\Lambda(P(z^*)) \subseteq \widehat{N}_{{\cal F} _D}(z^*)$ (\cite[Theorem 6.14]{Rockafellar}), if
\begin{equation}\label{sufficentc}
\nabla  P(z^*)^T \left(Q_1 ^\circ \cap ( \ker \nabla  P(z^*)^T+ Q_2^\circ) \right)\subseteq \nabla P(z^*)^T \widehat N_\Lambda(P(z^*)),
\end{equation}
then the equality holds in (\ref{tight}) and consequently,
\begin{equation}\widehat{N}_{{\cal F} _D}(z^*) =\nabla  P(z^*)^T \left(Q_1 ^\circ \cap ( \ker \nabla  P(z^*)^T+ Q_2^\circ) \right)=\nabla P(z^*)^T \widehat N_\Lambda(P(z^*)). \label{B-S} \end{equation}  How to choose $Q_1,Q_2$ satisfying the condition (\ref{sufficentc})?  One good choice is to find $Q_1,Q_2$ satisfying
$$Q_1^\circ \cap Q_2^\circ=\widehat{N}_\Lambda (P(z^*)),$$ since then condition (\ref{sufficentc}) holds whenever $\nabla P(z^*)$ has full rank.

Based on  the estimates of the regular normal cone in (\ref{tight}) and the fact that any local minimizer  is an B-stationary point, Benko and Gfrerer in \cite{Benko-Gfrerer2017} introduced the concept of the so-called $\mathcal{Q}$-stationarity. Moreover when an $\mathcal{Q}$-stationary point is also an M-stationary point, then they call it an $\mathcal{Q}_M$-stationary point.
In our case, since $T_\Lambda(P(z^*))$ is the union of finitely many convex polyheral sets, we can choose $Q_1,Q_2$ to be closed convex polyhedral cones.  By \cite[Proposition 1]{Benko-Gfrerer2017}, the polyhedrality of the cones $Q_i\subseteq T_\Lambda(P(z^*)), i=1,2$ ensures validity of (\ref{validity}).
 We now give definition for $\mathcal{Q}$-stationarity for the disjunctive program. }
\begin{defi}{\rm(\cite[Definiton 4 and Lemma 2]{Benko-Gfrerer2017})\label{definition of Q-statioanry}
Let $\mathcal{Q}$ denote some collection of pairs $(Q_1,Q_2)$ of closed convex polyhedral cones fulfilling
$Q_i\subseteq T_\Lambda(P(z^*)), i=1,2$.
\begin{itemize}
\item[(i)] Given $(Q_1,Q_2)\in\mathcal{Q}$, we say that $z^*$ is $\mathcal{Q}$-stationary with respect to $(Q_1,Q_2)$ for program (\ref{disjunctive constraints}) if
    \begin{equation}
    0\in\nabla f(z^*)+\nabla P(z^*)^T\left(Q_1^\circ\cap(\ker\nabla P(z^*)^T+Q_2^\circ)\right).\nonumber
    \end{equation}

\item[(ii)] We say that $z^*$ is $\mathcal{Q}$-stationary for program (\ref{disjunctive constraints}), if $z^*$ is $\mathcal{Q}$-stationary with respect to some pair $(Q_1,Q_2)\in\mathcal{Q}$.
\item [(iii)] We say that $z^*$ is $\mathcal{Q}_M$-stationary provided that  $z^*$ is both M-stationary and $\mathcal{Q}$-stationary with respect to some pair $(Q_1,Q_2)\in\mathcal{Q}$, i.e.,  there exists a pair $(Q_1,Q_2)\in\mathcal{Q}$ such that
    $$0\in\nabla f(z^*)+\nabla P(z^*)^T\left(Q_1^\circ\cap(\ker\nabla P(z^*)^T+Q_2^\circ)\cap N_\Lambda(P(z^*))\right).$$
\end{itemize}}
\end{defi}
Based on the discussion before Definition \ref{definition of Q-statioanry}, we obtain the following optimality conditions.
\begin{prop}\label{Prop3.3}
 Let $z^*$ be a local optimal solution for program (\ref{disjunctive constraints}). If  GGCQ holds at $z^*$, then $z^*$ is $\mathcal{Q}$-stationary with respect to every pair $(Q_1,Q_2)\in \mathcal{Q}$. {Moreover,}  $z^*$ 
is $\mathcal{Q}_M$-stationary with respect to every pair $(Q_1,Q_2)\in \mathcal{Q}$. Conversely, if $z^*$ is $\mathcal{Q}$-stationary with respect to some pair $(Q_1,Q_2)\in \mathcal{Q}$ fulfilling
(\ref{sufficentc}),  then $z^*$ is S-stationary and consequently also B-stationary.
\end{prop}
\begin{proof} Let $z^*$ be a local optimal solution for program (\ref{disjunctive constraints}). Then  by Proposition \ref{prop3.1}, $z^*$ is a B-stationary point, i.e.,
$-\nabla f(z^*)\in \widehat{N}_{{\cal F} _D} (z^*)$.   If  GGCQ holds at $z^*$, then by (\ref{tight}) and  Definition \ref{definition of Q-statioanry}, $z^*$ is $\mathcal{Q}$-stationary with respect to every pair $(Q_1,Q_2)\in \mathcal{Q}$.
{Moreover,} by Proposition \ref{prop3.1}, it is also an M-stationary point and hence $\mathcal{Q}_M$-stationary. Conversely, suppose that  $z^*$ is $\mathcal{Q}$-stationary with respect to some pair $(Q_1,Q_2)\in \mathcal{Q}$ fulfilling (\ref{sufficentc}). Then by definition  of $\mathcal{Q}$-stationarity
and (\ref{B-S}), $z^*$ is also S-stationary and B-stationary.
\end{proof}


Now we review the asymptotical version of M-stationarity.
Using a simple penalization argument,  \cite[Theorem 3.2]{Mehlitz2020}  showed that  any  local minimizer $z^*$ of MPDC must be  AM-stationary for MPDC. The question is under what conditions,  is an AM-stationary point  M-stationary?  In \cite[Definition 3.8]{Mehlitz2020}, Mehlitz defined the so-called asymptotically Mordukhovich-regularity (AM-regularity for short) and showed that under AM-regularity, an AM-stationary point is M-stationary.
Moreover, for the case of MPDC, according to the equivalence theorem shown in  \cite[Theorem 5.3]{Mehlitz2020}, we can define AM-{regularity} as follows.
\begin{defi}\label{AM-stationary}\rm\cite[Definition 3.1]{Mehlitz2020}
Let $z^*\in {{\cal F} _D}$. We say that $z^*$ is asymptotically M-stationary (AM-stationary)   if there exist sequences $\{z^k\},\,\{\varepsilon^k\}\subseteq \mathbb{R}^n$  with $z^k\rightarrow z^*,\varepsilon^k\rightarrow0$ such that
$$\varepsilon^k\in  \nabla f(z^k)+ \nabla P(z^k)^T N_\Lambda (P(z^k))\quad \forall k.$$ \end{defi}
\begin{defi}\label{AM-regular}\rm\cite[Theorem 5.3]{Mehlitz2020}
Let $z^*\in {{\cal F} _D}$. Define a set-valued mapping $\mathcal{K}:\mathbb{R}^n\rightrightarrows\mathbb{R}^n$ by means of $$\mathcal{K}(z):=\nabla P(z)^TN_\Lambda(P(z^*))\quad \forall z\in\mathbb{R}^n.$$
We say that  $z^*$ is AM-regular if  the following condition holds:
$$\limsup_{z\rightarrow z^*}\mathcal{K}(z)\subseteq \mathcal{K}(z^*),$$
where
\begin{eqnarray*}{
\limsup_{z\to z^*}\mathcal{K}(z):=\{y\in\mathbb{R}^n|\exists
\ z^k\to z^*,y^k\to y, \mbox{ s.t. } y^k\in \mathcal{K}(z^k)\ \forall k \}.}
\end{eqnarray*}
\end{defi}
\begin{prop}{{\rm(\cite[Theorem 3.2 and Theorem 3.9]{Mehlitz2020})}\label{AM-Prop}
Let $z^*$ be a local minimizer of MPDC. Then $z^*$ is AM-stationary. {Moreover,} suppose that  $z^*$ is AM-regular. Then $z^*$ is M-stationary.}
\end{prop}

Recently, the following  directional version of the limiting normal cone was introduced.
\begin{defi}[directional normal cones] \rm (\cite[Definition 2]{Gfrerer SVAA2013} or \cite[Definition 2.3]{Ginchev and Mordukhovich2011})
Let $A \subseteq \mathbb{R}^m$ be closed, $x^*\in A$ and $d\in\mathbb{R}^m$. The limiting normal cone to $A$ at $x^*$ in direction $d$ is defined by
$$N_A(x^*;d):=\left\{\zeta\in\mathbb{R}^m {\left|\exists t_k\downarrow0, d^k\rightarrow d,\zeta^k\rightarrow\zeta, \mbox{ s.t. }\zeta^k\in\widehat{N}_A(x^*+t_kd^k)\right.}\right\}.$$
\end{defi}

From the definition, it is obvious that the limiting normal cone to $A$ at $x^*$ in direction $d=0$ is equal to the limiting normal cone. It is also easy to see that  $N_A(x^*;d)=\emptyset$ if $d\notin T_A(x^*)$ and $N_A(x^*;d)\subseteq N_A(x^*)$  for all $d$.
If $A$ is convex, then by \cite[Lemma 2.1]{Gfrerer SIAM Optimal 2014}, the following relationship holds
\begin{equation}N_A(x^*;d)=N_A(x^*)\cap \{d\}^\perp \quad \forall d\in T_A(x^*).\label{convexcase}
\end{equation}
\begin{defi} \rm (\cite[Definition 3.3]{Ye-Zhou2018MP})
Let $A \subseteq \mathbb{R}^m$ be closed, $x^*\in A$ and $d\in\mathbb{R}^m$. We say that set  $A$ is directionally regular at $x^*$ if
$$N_A(x^*;d)=N_A^i(x^*;d) \quad \forall d,$$ where $N_A^i(x^*;d):=\left\{\zeta\in\mathbb{R}^m {\left|\forall t_k\downarrow0, \exists d^k\rightarrow d,\zeta^k\rightarrow\zeta, \mbox{ s.t. }\zeta^k\in\widehat{N}_A(x^*+t_kd^k)\right.}\right\}.$
If $A$ is directionally regular at any point $x\in A$, we say that the set $A$ is directionally regular.
\end{defi}

By \cite[Proposition 3.5]{Ye-Zhou2018MP}, any closed convex set is directionally regular.
The following calculus rules will be useful.
It is a special case of \cite[Proposition 3.3]{Ye-Zhou2018MP}.
\begin{prop}{\rm(\cite[Proposition 3.3]{Ye-Zhou2018MP})}\label{proposition}
Let $A:=A_1\times A_2\times\cdots\times A_l$, where $A_{i}\subseteq \mathbb{R}^{m_i}$ is closed  for $i=1,2,\cdots,l$ and $m=m_1+m_2+\cdots+m_l$. Consider a point $x^*=(x^*_1,\cdots,x^*_l)\in A$ and a direction $d=(d_1,\cdots,d_l)\in\mathbb{R}^m$. Moreover, suppose that all except at most one of $A_i$ for $i=1,\cdots,l$ are directionally regular at $x^*_i$, then \begin{eqnarray*}
T_A(x^*)= T_{A_1}(x^*_1)\times\cdots\times T_{A_l}(x^*_l),
\quad N_A(x^*;d)=  N_{A_1}(x^*_1;d_1)\times\cdots\times N_{A_l}(x^*_l;d_l).
\end{eqnarray*}
\end{prop}
\begin{defi}[directional metric subregularity]\label{definition of directional metric subregularity} \rm (\cite[Definition 2.1]{Gfrerer SVAA2013})
Let $F(z):=P(z)-\Lambda$ be a set-valued map induced by $P(z)\in\Lambda$. We say that the set-valued map $F$ is metrically  subregular in direction $d$ at $(z^*,0)\in \gph F$, where
$\gph F:=\{(z,y)|y\in F(z)\}$
is the graph of $F$, if there exist $\kappa>0$ and $\rho,\delta>0$ such that
$\dist_{F^{-1}(0)}(z)\leq\kappa \dist_{\Lambda}(P(z)), \ \forall z\in z^*+V_{\rho,\delta}(d),$
where $
V_{\rho,\delta}(d):=\left\{z\in B_\rho(0)|\|\|d\|z-\|z\|d\|\leq\delta\|z\| \|d\|\right\}
$
is the so-called directional neighborhood in direction $d$,
 and $F^{-1}(y):=\{z|y\in F(z)\}$ denotes the inverse of $F$ at $y$. If $d=0$ in the above definition, then we say $F$ is metrically subregular at $(z^*,0)$.
\end{defi}

According to \cite {Benko etal2019}, when  the disjunctive set $\Lambda:=\cup_{l=1}^N \Lambda^l, \Lambda^l=\Pi_{i=1}^m [a_i^l,b_i^l], $
where  $ a_i^l,b_i^l $ are given numbers with  $ a_i^l\leq b_i^l $,
with possibility of  $ a_i^l=-\infty$ and
$ b_i^l =+\infty$,
we call  (\ref{disjunctive constraints})  the ortho-disjunctive program.
  We now recall the following sufficient conditions for directional metric subregularity for the ortho-disjunctive program.
\begin{defi}\rm Let $z^*$ be a feasible solution to the ortho-disjunctive program.
\label{directional normality}
\textcolor{red}{}
\begin{itemize}
{\rm\item[(a)]  ({\cite[Corollary 5.1]{Benko etal2019}})} We say that the  quasi-normality holds at $z^*$ in direction {$d\in L_{{\cal F} _D}^{\rm lin}(z^*)$} 
if
there exists no $\zeta\neq0$ such that
\begin{eqnarray}\label{no-zeta}
0=\nabla P(z^*)^T\zeta ,\quad \zeta\in N_\Lambda(P(z^*);\nabla P(z^*)d),
\end{eqnarray}
and
{\begin{eqnarray*}
\exists d^k \to d  \ t_k\downarrow0\
s.t.\ \ \zeta_i(P_i(z^*+t_kd^k)-P_i(z^*))>0\ if\ \zeta_i\neq0.
\label{quasi}
\end{eqnarray*}}
We say that  the directional quasi-normality holds at $z^*$ if the  quasi-normality holds at $z^*$ in any direction {$d\in L_{{\cal F} _D}^{\rm lin}(z^*)$}.
{\rm\item[(b)]  ({\cite[Corollary 4.5]{Benko etal2019}})} We say that the  pseudo-normality holds at $z^*$ in direction
{$d\in L_{{\cal F} _D}^{\rm lin}(z^*)$} if
there exists no $\zeta\neq0$ such that (\ref{no-zeta}) holds and
{\begin{eqnarray*}
\exists d^k \to d \ t_k\downarrow0\
s.t.\ \langle\zeta,P(z^*+t_kd^k)-P(z^*)\rangle>0.
\label{pseudo}
\end{eqnarray*}}
We say that  the directional pseudo-normality holds at $z^*$ if the  pseudo-normality holds at $z^*$ in any direction {$d\in L_{{\cal F} _D}^{\rm lin}(z^*)$}.
{\rm\item[(c)] (\cite[Theorem 4.3]{Gfr132})} We say that the first-order sufficient condition for metric subregularity (FOSCMS) holds at $z^*$ in direction {$d\in L_{{\cal F} _D}^{\rm lin}(z^*)$} if
there exists no $\zeta\neq0$ such that
(\ref{no-zeta}) holds.
{\rm\item[(d)] (\cite[Theorem 4.3]{Gfr132})} Suppose $P$ is second-order differentiable at $z^*$. We say that the second-order sufficient condition for metric subregularity (SOSCMS) holds at $z^*$ in direction
{$d\in L_{{\cal F} _D}^{\rm lin}(z^*)$} if
there exists no $\zeta\neq0$ such that
(\ref{no-zeta}) and the following second-order condition hold
$$ {\sum_{i=1}^{{m}}  \zeta_i d^T \nabla^2 P_i(z^*)d  \geq 0.}$$
\end{itemize}
\end{defi}
%

{Note that the concepts of directional quasi/pseudo-normality for the ortho-disjunctive program as defined in Definition \ref{directional normality}  correspond  precisely to the ones introduced for the general set-constrained optimization problem in \cite{Bai-Ye-Zhang2019}.}
It was shown in {\cite{Benko etal2019,Gfr132}} that the following implication holds:
$${{\mbox{ FOSCMS in } d \Longrightarrow}\mbox{ SOSCMS in } d} \Longrightarrow \mbox{  pseudo-normality in }  d\Longrightarrow  \mbox{ quasi-normality in }  d .$$
We refer the reader to higher order sufficient condition for metric subregularity and other sufficient conditions for metric subregularity in \cite{Benko etal2019}.

The following result is a directional version of \cite[Corollary 4.1]{Bai-Ye-Zhang2019}.
\begin{prop}{\rm(\cite[Corollary 4.1]{Bai-Ye-Zhang2019})} Suppose that the quasi-normality holds at $z^*$ in  {$d\in L_{{\cal F} _D}^{\rm lin}(z^*)$}. Then the set-valued map $F(z):=P(z)-\Lambda$ is metrically subregular at $(z^*,0)$ in direction $d$.
\end{prop}

We now summarize some first and second order necessary optimality conditions for MPDC in the following propositions.
\begin{prop} \label{Thm3.1}{\rm(\cite[Theorems 3.3]{Gfrerer SIAM Optimal 2014})} Let $z^*$ be a local minimizer of problem (\ref{disjunctive constraints}) and $d\in {\cal C}(z^*)$,
where $
{\cal{C}}(z^*):=\{{d\in L_{{\cal F} _D}^{\rm lin}(z^*)}|\nabla f(z^*)d\leq0\}$  is the critical cone at $z^*$. If $F(z):=P(z)-\Lambda$ is metrically subregular at $(z^*,0)$ in direction $d$, then M-stationary condition in direction $d$ holds. That is, there exists $\zeta$ such that
\begin{equation} 0= \nabla f(z^*)+ \nabla P(z^*)^T \zeta ,\quad  \zeta\in N_\Lambda (P(z^*); \nabla P(z^*)d). \label{firstorder}
\end{equation}
Moreover, if $f$ and $P$ are twice differentiable at $z^*$, then  there exists $\zeta$ satisfying (\ref{firstorder}) such that the second order condition holds:
$$d^T\nabla^2_z\mathcal{L}(z^*,\zeta) d\geq0,$$
where $\mathcal{L}(z,\zeta) := f(z)+ P(z)^T\zeta $ is the Lagrangian.
\end{prop}

We also give a sufficient optimality condition based on S-stationary condition below. One may refer to \cite[Theorems 3.3]{Gfrerer SIAM Optimal 2014} for more general sufficient optimality condition.
\begin{prop}\label{prop3.6}{\rm(\cite[Theorems 3.3]{Gfrerer SIAM Optimal 2014}) or \cite[Theorem 4.3]{Mehlitz2020optimization}})
Let $z^*$ be a feasible solution for problem (\ref{disjunctive constraints}) where $f$ and $P$ are twice differentiable at $z^*$. Suppose for each $0\not =d \in {\cal C}(z^*)$, there exists an S-multiplier $\zeta$ satisfying S-stationary condition
$$ 0= \nabla f(z^*)+ \nabla P(z^*)^T \zeta, \quad \zeta\in \widehat{N}_\Lambda(P(z^*)) $$ and the second order  condition
$$d^T\nabla^2_z\mathcal{L}(z^*,\zeta) d> 0.$$
 Then there is a constant $C>0$ and $N(z^*)$ a neighborhood of $z^*$ such that the following quadratic growth condition is valid:
 $$f(z) \geq f(z^*) +C\|z-z^*\|^2 \qquad \forall z\in {{\cal F} _D}\cap N(z^*).$$
 In particular, $z^*$ is a strict local minimizer of MPDC.
\end{prop}
\begin{defi}\rm (see (31) in \cite{GfrererYeZhou})
\label{Defn3.6}
Let {$d\in L_{{\cal F} _D}^{\rm lin}(z^*)$}. We say that  LICQ in direction $d$ (LICQ($d$)) holds at $z^*$ if
$$\nabla P(z^*)^T \lambda =0,\quad  \lambda \in {\rm span } N_\Lambda (P(z^*);\nabla P(z^*) d) \Longrightarrow \lambda=0.$$
\end{defi}
\begin{prop}{\rm (\cite[Lemma 7]{GfrererYeZhou})}
\label{Thm3.2}
Let $z^*$ be a local minimizer of problem (\ref{disjunctive constraints}) and $d\in{\cal C}(z^*)$. Suppose that LICQ(d) holds. Then  S-stationary condition in direction $d$ holds. That is, there exists $\zeta$ such that
\begin{equation}0= {\nabla f(z^*)+} \nabla P(z^*)^T\zeta ,\quad \zeta\in \widehat{N}_{T_\Lambda(P(z^*))} (\nabla P(z^*)d).\label{directionalS}
\end{equation}
Moreover, the multiplier $\zeta$ is unique.
\end{prop}

\section{Optimality conditions for MPSC from the corresponding ones for MPDC}
In this section, we reformulate MPSC as the following disjunctive program and derive the corresponding optimality conditions from  Section \ref{sec3}.
\begin{eqnarray}\min f(z) \mbox{ s.t. } P(z)\in \Lambda,\label{equivalent-MPSC}\end{eqnarray}
where
\begin{equation}
P(z):=(g(z),h(z), (G(z),H(z))), \quad
\Lambda:=\mathbb{R}^p_-\times \{0\}^q \times \Omega_{SC}^m \label{definingfunctions}
\end{equation}
with the switching cone
\begin{eqnarray}\label{C-set}
\Omega_{SC}:=\{(a,b)\in\mathbb{R}^2|ab=0\}.
\end{eqnarray}
Since the switching cone $\Omega_{SC}$ is the union of the two subspaces $\mathbb{R}\times\{0\}$ and $\{0\}\times\mathbb{R}$, the cone $\Lambda$ is the union of $2^m$ convex polyhedral sets.

%
%
%
%
%
%
%
By a straightforward calculation,  we can obtain the formulas for various tangent and normal cones for  the switching cone $\Omega_{SC}$ defined in (\ref{C-set}) as follows.

\begin{lema}\label{normal cone lema}
For all $a=(a_1,a_2)\in \Omega_{SC}$ we have
\begin{eqnarray*}
T_{\Omega_{SC}}(a)=\left\{\begin{array}{ll}
\{0\}\times \mathbb{R} & if\ a_1=0,a_2\neq0,\\
{\Omega_{SC}}  & if\ a_1=a_2=0,\\
\mathbb{R} \times \{0\} & if\ a_1\neq0,a_2=0
\end{array}\right\},
\end{eqnarray*}
\begin{eqnarray*}
 \widehat{N}_{\Omega_{SC}}(a)=\left\{\begin{array}{ll}
\mathbb{R} \times \{0\}  & if\ a_1=0,a_2\neq0,\\
\{0\}\times \{0\} & if\ a_1=a_2=0,\\
\{0\}\times \mathbb{R}  & if\ a_1\neq0,a_2=0
\end{array}\right\},
\end{eqnarray*}
\begin{eqnarray*}
N_{\Omega_{SC}}(a) = \left\{\begin{array}{ll}
\mathbb{R} \times \{0\}  & if\ a_1=0,a_2\neq0,\\
{\Omega_{SC}}  & if\ a_1=a_2=0,\\
\{0\}\times \mathbb{R}  & if\ a_1\neq0,a_2=0
\end{array}\right\},
\end{eqnarray*}

\begin{eqnarray*}
\widehat{N}_{{T}_{\Omega_{SC}} (a)}(d)& =& \left\{\begin{array}{ll}
\mathbb{R} \times \{0\}  & if\ a_1=0,a_2\neq0  , d_1=0,\\
\{0\} \times \mathbb{R}  & if\ a_1\neq0, a_2=0  , d_2 =0,\\
\mathbb{R} \times \{0\}   & if\ a_1=a_2=0, d_1=0, d_2\not =0,\\
 \{0\} \times \mathbb{R} &  if\ a_1=a_2=0, d_1\not =0, d_2 =0,\\
\{0\}\times \{0\}& if \ a_1=a_2=d_1=d_2 =0 ,
\end{array}\right.
\end{eqnarray*}
\begin{eqnarray*}
N_{\Omega_{SC}} (a;d)=N_{\Omega_{SC}}^i (a;d)=\left\{\begin{array}{ll}
\mathbb{R} \times \{0\}  & if\ a_1=0,a_2\neq0, d_1=0,\\
\{0\} \times \mathbb{R}  & if\ a_1\neq0, a_2=0, d_2 =0,\\
\mathbb{R} \times \{0\}   & if\ a_1=a_2=0, d_1=0, d_2\not =0,\\
 \{0\} \times \mathbb{R} &  if\ a_1=a_2=0, d_1\not =0, d_2 =0,\\
{\Omega_{SC}} & if \ a_1=a_2=d_1=d_2 =0 .
\end{array}\right.
\end{eqnarray*}
Hence the switching cone $\Omega_{SC}$ is directionally regular.
\end{lema}

{Since $\mathbb{R}^p_-$ and  $ \{0\}^q $ are obviously directionally regular and the switching cone $\Omega_{SC}$ is directionally regular,  the  calculus rules for tangent and directional normal cones of $\Lambda$ as a Cartesian product in  Proposition \ref{proposition} hold. Hence for any $z^*$ such that $P(z^*)\in \Lambda$, we can obtain the expression for the tangent cone  to $\Lambda$ at $P(z^*)$ as follows:
\begin{eqnarray}\label{tangentcone}
{T_\Lambda(P(z^*))}= T_{\mathbb{R}^p_-}(g(z^*)) \times T_{\{0\}^q}( 0)\times \Pi_{i=1}^m T_{\Omega_{SC}} (G_i(z^*),H_i(z^*)).
\end{eqnarray}}

First we study  $\mathcal{Q}$ and $\mathcal{Q}_M$ stationary conditions for MPSC. Let $\mathcal{P}(\I_{GH}^*)$ be the set of all (disjoint) bipartitions of $\I_{GH}^*$. For fixed $(\beta_1,\beta_2)\in\mathcal{P}(\I_{GH}^*)$, we define the convex polyhedral cone
\begin{eqnarray*}
Q_{SC}^{\beta_1,\beta_2}:= T_{\mathbb{R}^p_-}(g(z^*)) \times \{0\}^q\times \prod_{i=1}^m \tau_i^{\beta_1,\beta_2}
\end{eqnarray*}
where $\tau_i^{\beta_1,\beta_2}:=T_{\Omega_{SC}}(G_i(z^*),H_i(z^*))$ if $i\in\I_G^*\cup\I_H^*$ and
\begin{eqnarray*}
\tau_i^{\beta_1,\beta_2}:=\left\{\begin{array}{cc}
\{0\}\times\mathbb{R}&  {\rm if}\ i\in\beta_1,\\
\mathbb{R}\times\{0\}& {\rm if}\ i\in \beta_2.\end{array}\right.
\end{eqnarray*}
{By (\ref{tangentcone}) and the formula for the tangent cone to $\Omega_{SC}$  in Lemma \ref{normal cone lema}, it is easy to see that $Q_{SC}^{\beta_1,\beta_2}$ is a subset of $T_\Lambda(P(z^*))$ as required by $\mathcal{Q}$-stationarity. Moreover, similarly to \cite[Lemma 3]{Benko-Gfrerer2017}, we can show that for any  $(\beta_1,\beta_2)\in\mathcal{P}(\I_{GH}^*)$,
$$(Q_{SC}^{\beta_1,\beta_2})^\circ \cap (Q_{SC}^{\beta_2,\beta_1})^\circ=\widehat{N}_\Lambda(P(z^*)).$$
Hence according to the discussion in Section \ref{sec3}, $Q_1:=Q_{SC}^{\beta_1,\beta_2}, Q_2:=Q_{SC}^{\beta_2,\beta_1}$ would be a good choice for the $\mathcal{Q}$-stationarity. Similar to \cite[Proposition 4]{Benko-Gfrerer2017}, we can derive the definition of $\mathcal{Q}$-stationarity for MPSC by using the corresponding definitions for the disjunctive program in Definition \ref{definition of Q-statioanry} by using $(Q_1,Q_2):=(Q_{SC}^{\beta_1,\beta_2},Q_{SC}^{\beta_2,\beta_1})$. By Definition \ref{definition of Q-statioanry}, $z^*$ is $\mathcal{Q}$-stationary with respect to $(Q_{SC}^{\beta_1,\beta_2},Q_{SC}^{\beta_2,\beta_1})$ if
\begin{eqnarray*}
-\nabla f(z^*)& \in  & \nabla P(z^*)^T\left((Q_{SC}^{\beta_1,\beta_2})^\circ \cap(\ker\nabla P(z^*)^T+(Q_{SC}^{\beta_2,\beta_1})^\circ\right)\\
&=&  \left \{\sum_{i=1}^p\lambda_i^g\nabla g_i(z^*)+\sum_{i=1}^q\lambda_i^h\nabla h_i(z^*)+\sum_{i=1}^m\left(\lambda_i^G\nabla G_i(z^*)+\lambda_i^H\nabla H_i(z^*)\right) \right . \\
&& \qquad\qquad \left| (\lambda^h,\lambda^g,\lambda^G,\lambda^H)\in  (Q_{SC}^{\beta_1,\beta_2})^\circ \cap(\ker\nabla p(z^*)^T+(Q_{SC}^{\beta_2,\beta_1})^\circ)  \right \}.
\end{eqnarray*}}

{Next we are aiming to find a formula for $(Q_{SC}^{\beta_1,\beta_2})^\circ \cap\left(\ker\nabla P(z^*)^T+(Q_{SC}^{\beta_2,\beta_1})^\circ\right) $ in the above.
 Obviously, we have $(Q_{SC}^{\beta_1,\beta_2})^\circ=N_{\mathbb{R}_-^p} (g(z^*))\times \mathbb{R}^q  \times   (\prod_{i=1}^m \tau_i^{\beta_1,\beta_2})^\circ$ and the set $(Q_{SC}^{\beta_1,\beta_2})^\circ\cap \left(\ker\nabla P(z^*)^T+(Q_{SC}^{\beta_2,\beta_1})^\circ\right) $  consists of all $\lambda=(\lambda^h, \lambda^g,\lambda^G,\lambda^H)$ such that there exists $\mu=(\mu^h,\mu^g,\mu^G,\mu^H) \in \ker \nabla P(z^*)^T$ and $(\eta^h,\eta^g,\eta^G,\eta^H) \in (Q_{SC}^{\beta_2,\beta_1})^\circ=N_{\mathbb{R}_-^p} (g(z^*))\times \mathbb{R}^q  \times   (\prod_{i=1}^m \tau_i^{\beta_2,\beta_1})^\circ$ such that
 $$\lambda=\mu +\eta\in (Q_{SC}^{\beta_1,\beta_2})^\circ.$$
We now analyse the following different cases.
\begin{itemize}
\item Equality constraints:  We obtain $\lambda^h=\mu^h+\eta^h\in \mathbb{R}^q, \eta^h \in \mathbb{R}^q$, i.e., $\lambda^h,\mu^h \in \mathbb{R}^q$.
\item Inequality constraints: For $i \in \I_g^*${,} we have $\lambda^g_i=\mu^g_i+\eta^g_i \geq 0, \eta_i^g \geq 0$ or equivalently $\lambda_i^g \geq \max \{0, \mu_i^g \}$, whereas for $i\in \{1,\dots, p\}\setminus \I_g^*$, we obtain $\lambda_i^g=\mu_i^g=0$.
\item  {$i\in \I_G^*$}: Since $(\tau_i^{\beta_1,\beta_2})^\circ=\mathbb{R}\times \{0\}$, we obtain $\lambda_i^H=\eta_i^H=0$ and so $\mu_i^H=0$.
\item  {$i\in \I_H^*$}: Similarly as in the previous case{,} we obtain $\lambda_i^G=\mu_i^G=0.$
\item $i\in \beta_1$: Since $(\tau_i^{\beta_1,\beta_2})^\circ=\mathbb{R}\times \{0\}$ and $(\tau_i^{\beta_2,\beta_1})^\circ=\{0\} \times \mathbb{R}$, we have
$$(\lambda_i^G,\lambda_i^H)=(\mu_i^G,\mu_i^H)+(\eta_i^G,\eta_i^H)\in \mathbb{R}\times \{0\}$$
and $(\eta_i^G,\eta_i^H)\in \{0\} \times \mathbb{R}$.  Equivalently, we obtain $\lambda_i^H=0$ and $\lambda_i^G=\mu_i^G$.
\item $i\in \beta_2$: Similarly as in the previous case{,} we obtain $\lambda_i^G=0$ and $\lambda_i^H=\mu_i^H$.
\end{itemize}
We now denote two multiplier sets
\begin{eqnarray*}
\begin{array}{rr}\mathcal{R}_{SC}:=\{(\mu^g,\mu^h,\mu^G,\mu^H)
\in\mathbb{R}^p\times\mathbb{R}^q\times\mathbb{R}^m\times\mathbb{R}^m| \mu_i^g=0,i=\{1,\cdots,p\}\setminus\I_g^*,\quad \\
\mu_i^G=0,i\in\I_H^*,\mu_i^H=0,i\in\I_G^* \}\end{array}
\end{eqnarray*}}
and
\small{\begin{eqnarray*}
\mathcal{N}_{SC}:=
\left\{(\mu^g,\mu^h,\mu^G,\mu^H)\in \mathcal{R}_{SC}\left|\sum_{i=1}^p\mu_i^g\nabla g_i(z^*)+\sum_{i=1}^q\mu_i^h\nabla h_i(z^*)+\sum_{i=1}^m(\mu_i^G\nabla G_i(z^*)+\mu_i^H\nabla H_i(z^*))=0\right.\right\}.
\end{eqnarray*}}
Based on the discussion above, we can now give the following definition.
\begin{defi}\rm
Let $z^*\in\F$.
We say that $z^*$ is $\mathcal{Q}$-stationary for MPSC (\ref{MPSC}) with respect to $(\beta_1,\beta_2)\in\mathcal{P}(\I_{GH}^*)$ if  there exists  multipliers  $(\lambda^g,\lambda^h,\lambda^G,\lambda^H)\in \mathcal{R}_{SC}$ such that
$$0=\nabla f(z^*)+\sum_{i=1}^p\lambda_i^g\nabla g_i(z^*)+\sum_{i=1}^q\lambda_i^h\nabla h_i(z^*)+\sum_{i=1}^m(\lambda_i^G\nabla G_i(z^*)+\lambda_i^H\nabla H_i(z^*))$$
and $(\mu^g,\mu^h,\mu^G,\mu^H)\in \mathcal{N}_{SC}$
such that
$$ \lambda_i^g\geq\max\{\mu_i^g,0\},i\in\I_g^*,\lambda_i^H=0,
\lambda^G_i=\mu^G_i, i\in\beta_1,\lambda_i^G=0,\lambda^H_i=\mu^H_i, i\in\beta_2.$$



\end{defi}

It is easy to see that for MPSC,  an $\mathcal{Q}$-stationary point  is also {an M-stationary point}. Hence  for MPSC, $\mathcal{Q}$-stationary condition coincides with  $\mathcal{Q}_M$-stationary condition.

Applying Proposition \ref{Prop3.3}, similar to \cite[Proposition 4 and Theorem 5]{Benko-Gfrerer2017}, we have the following optimality conditions.
\begin{prop}Let $z^*$ be a local optimal solution for MPSC. If GGCQ holds at $z^*$, then $z^*$ is $\mathcal{Q}$-stationary with respect to every pair $(\beta_1,\beta_2)\in\mathcal{P}(\I_{GH}^*)$.  Conversely, if $z^*$ is $\mathcal{Q}$-stationary with respect to some pair $(\beta_1,\beta_2)\in\mathcal{P}(\I_{GH}^*)$ such that for every $\mu\in \mathcal{N}_{SC}$ there holds
\begin{eqnarray}
&&\mu_i^G\mu_{i'}^G=0,\mu_i^H\mu_{i'}^H=0\quad\forall(i,i')\in\beta_1\times\beta_2,\label{propassum1}\\
&&\mu_i^G\mu_{i'}^H=0\quad\forall(i,i')\in\beta_1\times\beta_1,\label{propassum2}\\
&&\mu_i^G\mu_{i'}^H=0\quad\forall(i,i')\in\beta_2\times\beta_2.\label{propassum3}
\end{eqnarray}
Then $z^*$ is S-stationary.
\end{prop}
\begin{proof} {The first statement follows directly   from Proposition \ref{Prop3.3}. We now prove the converse statement.} Suppose that $z^*$ is $\mathcal{Q}$-stationary with respect to some pair $(\beta_1,\beta_2)\in\mathcal{P}(\I_{GH}^*)$. Then by definition, there exist $(\lambda^g,\lambda^h,\lambda^G,\lambda^H)\in \mathcal{R}_{SC}$ and $(\mu^g,\mu^h,\mu^G,\mu^H)\in \mathcal{N}_{SC}$  satisfying the $\mathcal{Q}$-stationary condition.
{Since by the definition of $\mathcal{Q}$-stationarity, $\lambda_i^H=0, i\in \beta_1$ and $\lambda_i^G=0, i\in \beta_2$. So if $\lambda_i^G=0, i\in \beta_1$ and $\lambda_i^H=0, i\in \beta_2$, then $z^*$ must be  S-stationary in this case.} Otherwise, either there is some $j\in \beta_1$ such that $\lambda_j^G\not =0$ or some $j\in \beta_2$ such that $\lambda_j^H\not =0$. First consider the case when $\lambda_j^G\not =0$ for some $j\in \beta_1$. Set $(\tilde \lambda^g,\tilde\lambda^h,\tilde\lambda^G,\tilde\lambda^H):=(\lambda^g-\mu^g,\lambda^h-\mu^h,\lambda^G-\mu^G,\lambda^H-\mu^H)$. Then
$$0=\nabla f(z^*)+\sum_{i=1}^p\tilde  \lambda_i^g\nabla g_i(z^*)+\sum_{i=1}^q \tilde \lambda_i^h\nabla h_i(z^*)+\sum_{i=1}^m(\tilde \lambda_i^G\nabla G_i(z^*)+\tilde \lambda_i^H\nabla H_i(z^*))$$
and
{$$\tilde  \lambda_i^g= 0, i\notin \I_g^*, \tilde  \lambda_i^g\geq 0, i\in \I_g^*, \tilde \lambda_i^G=0,i\in\I_H^*,\tilde \lambda_i^H=0,i\in\I_G^*, \tilde \lambda_i^G=0, i\in \beta_1, \tilde \lambda_i^H=0, i\in \beta_2.$$}Further, since $0\not = \lambda_j^G=\mu_j^G$, {then by (\ref{propassum1}) we have $\mu_i^G=0 \ \forall i\in \beta_2$  and by (\ref{propassum2}) we have $\mu_i^H=0\  \forall i\in \beta_1$ .} Consequently, $\tilde \lambda_i^G=0$, and {$\tilde\lambda_i^H=0$}  holds for all $i\in \beta_1\cup \beta_2$. Hence $z^*$ is  S-stationary. The proof for the case when $\lambda_j^H\not =0$ for some $j\in \beta_2$ is similar and  (\ref{propassum1}) and  (\ref{propassum3}) are used to derive the result in this case.
\end{proof}
%
%

{Applying Definition \ref{AM-stationary}-\ref{AM-regular} and Lemma \ref{normal cone lema}, we have the following AM-stationary condition and AM-regular for MPSC.
\begin{defi}\rm
Let $z^*\in\F$. We say that $z^*$ is AM-stationary of MPSC if there exist sequences $\{z^k\}\subseteq\F,\,\{\varepsilon^k\}\subseteq\mathbb{R}^n$  and multipliers $\left\{(\lambda^{g,k},\lambda^{h,k},\lambda^{G,k},\lambda^{H,k})\right\}\subseteq \mathbb{R}^p\times\mathbb{R}^q\times\mathbb{R}^m\times\mathbb{R}^m$ with $\varepsilon^k\rightarrow0,z^k\rightarrow z^*$ such that $\forall k$
\begin{eqnarray*}
&&\nabla f(z^k)+\sum_{i=1}^p\lambda^{g,k}_i\nabla g_i(z^k)+\sum_{i=1}^q\lambda^{h,k}_i\nabla h_i(z^k)+\sum_{i=1}^m(\lambda^{G,k}_i\nabla G_i(z^k)+\lambda^{H,k}_i\nabla H_i(z^k))=\varepsilon^k,\\
&&\lambda_i^{g,k}=0,if\ g_i(z^k)<0;\lambda^{g,k}_i\geq 0\ if\ g_i(z^k)=0;\
\lambda^{G,k}_i=0,\ if\ H_i(z^k)=0\neq G_i(z^k);\\
&&\lambda^{H,k}_i=0,\ if\ G_i(z^k)=0\neq H_i(z^k);
\ \lambda^{G,k}_i\lambda^{H,k}_i=0,\ if\ G_i(z^k)=H_i(z^k)=0.
\end{eqnarray*}
\end{defi}
\begin{defi}\rm
Let $z^*\in\F$. Define a set-valued mapping $\mathcal{K}:\mathbb{R}^n\rightrightarrows\mathbb{R}^{n}$ by means of
\begin{eqnarray*}
\mathcal{K}(z):=\left\{\begin{array}{ll} \sum_{i=1}^p\lambda^{g}_i\nabla g_i(z)+\sum_{i=1}^q\lambda^{h}_i\nabla h_i(z)\\+\sum_{i=1}^m(\lambda^{G}_i\nabla G_i(z)+\lambda^{H}_i\nabla H_i(z))\end{array}\left|\begin{array}{ll}(\lambda^{g},\lambda^{h},\lambda^{G},\lambda^{H})\}\subseteq\mathbb{R}^p_+\times\mathbb{R}^q\times\mathbb{R}^m\times\mathbb{R}^m,\\
\lambda_i^g=0\ {\rm for}\ i\notin\I_g^*,\\
\lambda_i^H=0\ {\rm for}\ i\in\I_G^*,\lambda_i^G=0\ {\rm for}\ i\in\I_H^*,\\
\lambda_i^G\lambda_i^H=0\ {\rm for}\ i\in\I_{GH}^*\end{array}\right.\right\}.
\end{eqnarray*}
We say that  $z^*$ is AM-regular if  the following condition holds:
$$\limsup_{z\rightarrow z^*}\mathcal{K}(z)\subseteq \mathcal{K}(z^*).$$
\end{defi}}
{Applying Proposition \ref{AM-Prop}, we have the following conclusion.
\begin{thm}Let $z^*$ be a local minimizer of MPSC, then $z^*$ is AM-stationary. {Moreover,} suppose that  $z^*$ is AM-regular. Then $z^*$ is M-stationary.
\end{thm}}
We now apply Propositions \ref{Thm3.1} and \ref{Thm3.2}  to MPSC in the form of (\ref{equivalent-MPSC}).
{By the expressions for $T_{\Lambda}(P(z^*))$  in (\ref{tangentcone}) and the expression for the tangent cone of the switching set in Lemma  \ref{normal cone lema}}, the
 linearization cone  of the feasible region ${\cal F}$ can be expressed as follows:
 \begin{eqnarray*}
L_{\cal F}^{\rm lin}(z^*)&:=& \big\{d|\nabla P(z^*)d\in T_{\Lambda}(P(z^*)) \big \}\\
&=&\left\{d\in\mathbb{R}^n{\left|\begin{array}{ll}
\nabla g_i(z^*)d\leq0 & i\in \I_g^*,\\
\nabla h_j(z^*)d=0 & j=1,\cdots,q,\\
\nabla G_i(z^*)d=0 & i\in\I_G^*,\\
\nabla H_i(z^*)d=0 & i\in\I_H^*,\\
{(\nabla G_i(z^*)d)(\nabla H_i(z^*)d)=0}  & i\in\I_{GH}^*
\end{array}\right.}\right\}.
\end{eqnarray*}

 Denote the critical cone  at $z^*$ by
$
{\cal{C}}_{\cal F}(z^*):=\{d\in L^{\rm lin}_\F(z^*)|\nabla f(z^*)d\leq0\}.
$
%
Given $d\in L_{\cal F}^{\rm lin}(z^*)$, we define
\begin{eqnarray*}
&&{\cal I}_g^*(d):=\left \{i\in{\cal I}_g^*\ |\ \nabla g_i(z^*)d=0 \right \}, \\
&&{\cal I}_G^*(d):=\left \{i\in{\cal I}_{GH}^*\ |\ \nabla G_i(z^*)d=0,\ \nabla H_i(z^*)d\neq 0 \right  \},\\
&&{\cal I}_H^*(d):= \left \{i\in{\cal I}_{GH}^*\ |\ \nabla G_i(z^*)d\neq 0,\ \nabla H_i(z^*)d=0 \right \},\\
&&{\cal I}_{GH}^*(d):=\left \{i\in{\cal I}_{GH}^*\ |\ \nabla G_i(z^*)d=\nabla H_i(z^*)d=0 \right  \}.
\end{eqnarray*}
Then by the Cartesian product rule in  Proposition \ref{proposition}, the expressions for the tangent cone and the  directional limiting normal cone to the switching cone in Lemma  \ref{normal cone lema}  we have,
\begin{eqnarray*}
N_\Lambda(P(z^*);\nabla P(z^*)d)
&=& N_{\mathbb{R}^p_-}(g(z^*);  \nabla g(z^*)d) \times N_{\{0\}^q}( h(z^*);\nabla h(z^*)d)\\&&\times \Pi_{i=1}^m N_{\Omega_{SC}} ((G_i(z^*),H_i(z^*));(\nabla G_i(z^*)d, \nabla H_i(z^*)d)),\end{eqnarray*}
with
\begin{eqnarray*}
N_{\Omega_{SC}} ((G_i(z^*),H_i(z^*));(\nabla G_i(z^*)d, \nabla H_i(z^*)d)) = \left\{( \lambda^G,\lambda^H)\left |\begin{array}{ll}
\lambda_i^G=0 & i\in\I_H^*\cup \I_H^*(d),\\
\lambda_i^H=0 & i\in\I_G^*\cup \I_G^*(d),\\
\lambda_i^G\lambda_i^H=0 & i\in\I_{GH}^*(d)
\end{array}\right. \right\}.
\end{eqnarray*}
Since $d\in L_{\cal F}^{\rm lin}(z^*)$ implies $\nabla g_i(z^*) d\leq 0 \ \forall i\in \I_g^*$ and  by (\ref{convexcase})$$N_{\mathbb{R}^p_-}(g(z^*);  \nabla g(z^*)d)=N_{\mathbb{R}^p_-}(g(z^*)) \cap \{\nabla g(z^*)d\}^\perp,$$
for any $\lambda^g\in N_{\mathbb{R}^p_-}(g(z^*);  \nabla g(z^*)d)$, we have that $\lambda_i^g =0 \ \forall i\not \in
\I_g^*(d)$ and $\lambda_i^g \geq 0 \ \forall i \in
\I_g^*(d)$.
Hence
\begin{eqnarray}
{N_\Lambda(P(z^*);\nabla P(z^*)d)} 
= \left\{( \lambda^g,\lambda^h,\lambda^G,\lambda^H)\left |\begin{array}{ll}
\lambda_i^g\geq 0 & i\in \I_g^*(d),\  \lambda_i^g= 0,  \  i \not \in \I_g^*(d),\\
\lambda_i^G=0 & i\in\I_H^*\cup \I_H^*(d),\\
\lambda_i^H=0 & i\in\I_G^*\cup \I_G^*(d),\\
\lambda_i^G\lambda_i^H=0 & i\in\I_{GH}^*(d)
\end{array}\right. \right\}.\label{Ncone}
\end{eqnarray}
Based on the {directional} M-stationary condition  (\ref{firstorder}) and directional S-stationary condition (\ref{directionalS}),
we now define the directional version of the W, S, M-stationarity for MPSC.
\begin{defi}\rm\label{doptimalitycondition}  Let $z^*$ be a feasible solution of MPSC and $d \in {\cal C}_{\cal F}(z^*)$.
We say that $z^*$ is a  W-stationary point of MPSC (\ref{MPSC})  in direction $d$ if there exists $(\lambda^g,\lambda^h,\lambda^G,\lambda^H)$ such that
\begin{eqnarray}
&& 0=\nabla f(z^*)+\sum_{i\in\I_g^*(d)}\lambda^g_i\nabla g_i(z^*)+\sum_{i=1}^q\lambda^h_i\nabla h_i(z^*)+\sum_{i=1}^m(\lambda^G_i\nabla G_i(z^*)+\lambda^H_i\nabla H_i(z^*)), \qquad  \label{W-stationary1-d}\\
&&\lambda^g_i\geq 0,\ i\in\I_g^*(d),\ \lambda^G_i=0,\ i\in\I_H^*\cup\I_H^*(d),\  \lambda^H_i=0,\ i\in\I_G^*\cup\I_G^*(d).\label{W-stationary4-d}
\end{eqnarray}


We say that $z^*$ is a M-stationary point of MPSC (\ref{MPSC}) in direction $d$ if there exists  $(\lambda^g,\lambda^h,\lambda^G,$ $\lambda^H)$ such that (\ref{W-stationary1-d})--(\ref{W-stationary4-d}) hold and
$\lambda^G_i\lambda^H_i=0,\ i\in\I_{GH}^*(d).$

We say that $z^*$ is a  S-stationary point of MPSC (\ref{MPSC}) in direction $d$  if there exists  $(\lambda^g,\lambda^h,\lambda^G,$ $\lambda^H)$ such that (\ref{W-stationary1-d})--(\ref{W-stationary4-d}) hold and
$\lambda^G_i=\lambda^H_i=0,\ i\in\I_{GH}^*(d).$
\end{defi}

Using the formula in (\ref{Ncone}), we have
\begin{eqnarray*}
{\rm span }N_\Lambda(P(z^*);\nabla P(z^*)d) = \left\{( \lambda^g,\lambda^h,\lambda^G,\lambda^H)
\left |\begin{array}{ll}
  \lambda_i^g= 0 &  \  i \not \in \I_g^*(d),\\
\lambda_i^G=0 & i\in\I_H^*\cup \I_H^*(d),\\
\lambda_i^H=0 & i\in\I_G^*\cup \I_G^*(d).
\end{array}\right.\right\}.
\end{eqnarray*}
Hence based on Definition \ref{Defn3.6}, we define the following directional version of the MPSC-LICQ.
\begin{defi}\rm\label{definition of MPSC LICQ(d)}
Let $z^*$ be a feasible solution of MPSC (\ref{MPSC}) and $d\in L_{\cal F}^{\rm lin}(z^*)$. We say that the MPSC-LICQ  in direction $d$ (MPSC-LICQ({$d$})) holds at $z^*$ if and only if the gradients \begin{eqnarray*}
\{\nabla g_i(z^*)|i\in\I_g^*(d)\}\cup\{\nabla h_j(z^*)|j=1,2,\cdots,q\}\cup\{\nabla G_i(z^*)|i\in\I_G^*\cup\I_G^*(d)\cup\I_{GH}^*(d)\}\\\cup\{\nabla H_i(z^*)|i\in\I_H^*\cup\I_H^*(d)\cup\I_{GH}^*(d)\}
\end{eqnarray*}
are linearly independent.
\end{defi}

Since $\I_g^*(0)=\I_g^*$, $\I_G^*(0)=\I_H^*(0)=\emptyset$ and $\I_{GH}^*(0)=\I_{GH}^*$,  it is easy to see that MPSC-LICQ(0) is exactly the MPSC-LICQ.

It is easy to see that MPSC is an ortho-disjunctive program. Hence by Definition \ref{directional normality}, 
it is easy to see that the  directional {quasi/pseudo-normality } for constraint system of MPSC (\ref{MPSC}) can be rewritten in the following form.

\begin{defi}\label{Defi4.3}\rm\label{definition of MPSC directional normality}
Let $z^*$ be a feasible solution of MPSC (\ref{MPSC}). $z^*$ is said to be  MPSC quasi- or pseudo-normal in  direction $ d\in L^{\rm lin}_\F(z^*)$ if  there exists no $(\lambda^g,\lambda^h,\lambda^G,\lambda^H)\neq0$ such that
\begin{itemize}
\item[\rm(i)]$0=\nabla g(z^*)^T\lambda^g+\nabla h(z^*)^T\lambda^h+\nabla G(z^*)^T\lambda^G+\nabla H(z^*)^T\lambda^H$;
\item[\rm(ii)]$\lambda^g_i\geq0,i\in\I_g^*(d);\lambda^g_i=0,i\notin\I^*_g(d)$;
$\lambda^H_i=0,i\in\I^*_G\cup\I^*_G(d)$;$\lambda^G_i=0,i\in\I^*_H\cup\I^*_H(d)$;
$\lambda^G_i\lambda^H_i=0,i\in\I^*_{G,H}(d)$;
\item[\rm(iii)]$\exists d^k\rightarrow d$ and $t_k\downarrow0$ such that
\begin{eqnarray*}
\left\{\begin{array}{ll}
\lambda^g_ig_i(z^*+t_kd^k)>0,\ {\rm if}\ \lambda^g_i\neq0,\\
\lambda^h_ih_i(z^*+t_kd^k)>0,\ {\rm if}\ \lambda^h_i\neq0,\\
\lambda^G_iG_i(z^*+t_kd^k)>0, \ {\rm if}\ \lambda^G_i\neq0,\\
\lambda^H_iH_i(z^*+t_kd^k)>0,\ {\rm if}\ \lambda^H_i\neq0,
\end{array}\right.
\end{eqnarray*}
or
\begin{eqnarray*}
{\lambda^g}^Tg(z^*+t_kd^k)+
{\lambda^h}^Th(z^*+t_kd^k)+
{\lambda^G}^TG(z^*+t_kd^k)+
{\lambda^H}^TH(z^*+t_kd^k)>0,
\end{eqnarray*}
respectively.
\end{itemize}
$z^*$ is said to be directionally quasi- or pseudo-normal if it is quasi- or pseudo-normal in all directions from $L^{\rm lin}_\F(z^*)$.
\end{defi}

{Note that MPSC quasi/pseudo-normality in direction $d=0$ coincides with  MPSC quasi/pseudo-normality defined as in \cite{Li-Guo}} and when $d\not =0$, the directional one is weaker.

We now apply Definition \ref{directional normality} to obtain FOSCMS/SOSCMS for MPSC.
\begin{defi}\label{Defi4.4} \rm
Let $z^*$ be a feasible solution of MPSC (\ref{MPSC}) and $d\in L_{\cal F}^{\rm lin}(z^*)$. We say that  MPSC first
order sufficient condition for metric subregularity (MPSC-FOSCMS) in direction $d$ holds at $z^*$ if   there exists no $(\lambda^g,\lambda^h,\lambda^G,\lambda^H)\neq0$ such that (i)--(ii) in Definition \ref{definition of MPSC directional normality} holds.
\end{defi}

{Note that  MPSC-FOSCMS in direction $d=0$ coincides with the MPSC-NNAMCQ defined as in \cite{Mehlitz MP}  and when $d\not =0$, MPSC-FOSCMS is weaker than MPSC-NNAMCQ.}
\begin{defi}\rm\label{second-order sufficient condition}
Let $z^*$ be a feasible solution of MPSC (\ref{MPSC}) and $d\in L_{\cal F}^{\rm lin}(z^*)$. We say that  MPSC second-order sufficient condition for metric subregularity (MPSC-SOSCMS) in direction $d$ holds at $z^*$ if there exists no $(\lambda^g,\lambda^h,\lambda^G,\lambda^H)\neq0$ such that  {(i)--(ii) in Definition \ref{definition of MPSC directional normality} hold}
and
{ $$d^T\nabla^2{\cal L}^0(z^*, \lambda^g,\lambda^h,\lambda^G,\lambda^H) d\geq0,$$
where ${\cal L}^0(z, \lambda^g,\lambda^h,\lambda^G,\lambda^H):=\langle  \lambda^g, g(z) \rangle +\langle  \lambda^h, h(z) \rangle+ \langle  \lambda^G, G(z) \rangle+\langle  \lambda^H, H(z) \rangle.$}
\end{defi}

 The following result follows from Propositions \ref{Thm3.1}-\ref{Thm3.2}. {The reader is referred to Figure 3 for sufficient conditions for MPSC quasi-normality.
\begin{thm}\label{stationary-d}
Let $z^*$ be a local minimizer for MPSC (\ref{MPSC}) and let $d\in {\cal {C}}_{\cal F}(z^*)$. If MPSC-LICQ(d) holds, then $z^*$ is an S-stationary point in direction $d$. If {MPSC quasi-normality}  holds at $z^*$ in direction $d$, then $z^*$ is an M-stationary point in direction $d$.
If $f$ and $F$ are twice  differentiable at $z^*$ then there exist an M-multiplier in direction $d$ denoted by $(\lambda^g,\lambda^h,\lambda^G,\lambda^H)$ such that the second-order condition holds:
$$d^T\nabla^2_z\mathcal{L}(z^*,\lambda^g,\lambda^h,\lambda^G,\lambda^H)d\geq0,$$
where ${\cal L}(z, \lambda^g,\lambda^h,\lambda^G,\lambda^H):=f(z)+\langle  \lambda^g, g(z) \rangle +\langle  \lambda^h, h(z) \rangle+ \langle  \lambda^G, G(z) \rangle+\langle  \lambda^H, H(z) \rangle.$
Conversely, suppose that $z^*$ is a feasible solution to MPSC and for each $0\not = d\in {\cal {C}}_{\cal F}(z^*)$, there is an S-multiplier denoted by $(\lambda^g,\lambda^h,\lambda^G,\lambda^H)$ and the second-order  condition
$$d^T\nabla^2_z\mathcal{L}(z^*,\lambda^g,\lambda^h,\lambda^G,\lambda^H)d> 0$$
holds, then $z^*$ is a strict local minimizer of MPSC.
\end{thm}


For MPEC, Gfrerer \cite{Gfrerer SIAM Optimal 2014} pointed out that the extended M-stationary condition (which means the directional M-stationary condition holds at every critical direction)  is usually hard to verify and introduced the strong M-stationary condition to build a bridge between M-stationarity and S-stationarity. Similarly we can propose a concept of strong M-stationary condition in a critical direction.
In what follows we denote by $r(z^*{;d})$ the rank of the family of gradients
\begin{eqnarray*}\label{family of gradients}
\begin{array}{cc}\{\nabla g_i(z^*)|i\in {\I_g^*(d)}\}\cup\{\nabla h_j(z^*)|j=1,\cdots,q\}\cup\{\nabla G_i(z^*)|i\in\I_G^*\cup \I_G^*(d) \cup \I_{GH}^*(d)\}\\
\cup\{\nabla H_i(z^*)|i\in\I_H^* \cup \I_H^*(d) \cup\I_{GH}^*(d)\}.
\end{array}
\end{eqnarray*}

\begin{defi}\rm
A triple of index sets $(J_g,J_G,J_H)$ with  $J_g\subseteq\I_g^*(d),J_G\subseteq\I_G^*\cup \I_G^*(d) \cup \I_{GH}^*(d),J_H\subseteq\I_H^* \cup \I_H^*(d) \cup\I_{GH}^*(d)$ is called an MPSC working set in direction $d$ for MPSC (\ref{MPSC}), if $J_G\cup J_H=\{1,2,\cdots,m\}$,
$$|J_g|+q+|J_G|+|J_H|=r(z^*{;d}),$$
and the family of gradients
\begin{eqnarray*}
\{\nabla g_i(z^*)|i\in J_g\}\cup\{\nabla h_j(z^*)|j=1,\cdots,q\}\cup\{\nabla G_i(z^*)|i\in J_G\}
\cup\{\nabla H_i(z^*)|i\in J_H\}
\end{eqnarray*}
is linearly independent.

The point $z^*$ is called strongly M-stationary in direction $d$ for  MPSC (\ref{MPSC}), if there exists an MPSC working set $(J_g,J_G,J_H)$ in direction $d$  together with  $\lambda=(\lambda^g,\lambda^h,\lambda^G,\lambda^H)$, an M-multiplier in direction $d$,  satisfying
\begin{eqnarray*}
&&\lambda^g_i=0,i\in\{1,\cdots,p\}\setminus J_g,\label{Strongly M 1}\\
&&\lambda^G_i=0,i\in\{1,\cdots,m\}\setminus J_G,\label{Strongly M 2}\\
&&\lambda^H_i=0,i\in\{1,\cdots,m\}\setminus J_H,\label{Strongly M 3}\\
&&\lambda^G_i=\lambda^H_i=0,i\in J_G\cap J_H.\label{Strongly M 4}
\end{eqnarray*}
\end{defi}

Similarly as in \cite[Theorem 4.3]{Gfrerer SIAM Optimal 2014}, we have the following result.
\begin{thm}
Assume that $z^*$ is M-stationary in direction $d\in {\cal C}_{\cal F}(z^*)$ for MPSC (\ref{MPSC}) and assume that there exists some MPSC working set in direction $d$. Then,  $z^*$ is strongly M-stationary in direction $d$.
\end{thm}

\begin{thm}
Let $z^*$ be feasible for MPSC (\ref{MPSC}) and assume that MPSC-LICQ{(d)} is fulfilled at $z^*$. Then $z^*$ is strongly M-stationary in direction $d$ if and only if it is S-stationary in direction $d$.
\end{thm}
\begin{proof}
The statement follows immediately from the fact that under MPSC-LICQ{({$d$})} there exists exactly one MPSC working set and this set fulfills $J_g=\I_g^*(d),J_G=\I_G^*\cup \I_G^*(d)\cup \I_{GH}^*(d),J_H=\I_H^*\cup\I_H^*(d)\cup \I_{GH}^*(d)$.
\end{proof}

In \cite[Example 5.2]{Mehlitz MP}, it was shown that the optimal solution of the following problem is M-stationary but not S-stationary. But we can  show that MPSC-LICQ({$d$}) holds at $z^*$ and $z^*$  is S-stationary in any nonzero critical direction.
\begin{ex}{\rm\cite[Example 5.2]{Mehlitz MP}}
Consider the following optimization problem
\begin{eqnarray*}
\min&&z_1+z_2^2\\
{\rm s.t.}&&-z_1+z_2\leq0,\quad z_1z_2=0.
\end{eqnarray*}
Its unique global minimizer is given by $z^*:=(0,0)$. The linearization cone and critical cone  of this problem at $z^*$ are given by
\begin{eqnarray*}
L^{\rm lin}_\F(z^*)&=&\{d\in\mathbb{R}^2|-d_1+d_2\leq0,d_1d_2=0\},\\
\mathcal{C}_{\cal F}(z^*)&=&\{d\in\mathbb{R}^2|-d_1+d_2\leq0,d_1d_2=0,d_1\leq0\}=\{d\in\mathbb{R}^2|d_1=0,d_2\leq0\}.
\end{eqnarray*}
Define $g(z):=-z_1+z_2, {G(z):=z_1, H(z):=z_2}$. Let $0\neq d\in\mathcal{C}_\F(z^*)$, then $\I_g^*(d),\I_H^*(d),\I_{GH}^*(d)$ are all empty but the index set $\I_G^*(d)=\{1\}$.
Hence MPSC-LICQ(d) holds at $z^*$. It is easy to check that $z^*$ is indeed S-stationary in any direction $0\neq d\in\mathcal{C}_{\cal F}(z^*)$.
\end{ex}

The strong M-stationarity in direction $d$ builds a bridge between M-stationarity in direction $d$ and S-stationarity in direction $d$. We summarize the relations among the various stationarity concepts in Figure 2.
%
%
%
\begin{figure}%
\centering
		\scriptsize
		 \tikzstyle{format}=[rectangle,draw,thin,fill=white]
		 \tikzstyle{test}=[diamond,aspect=2,draw,thin]
		\tikzstyle{point}=[coordinate,on grid,]
		\begin{tikzpicture}
        \node[format](S){S-stationary};
        \node[format,below of=S,node distance=10mm](S direction){S-stationary in direction $d$};
        \node[format,below of=S direction,node distance=7mm](Strong M direction){strongly M-stationary in direction $d$};
		\node[format,below of=Strong M direction,node distance=7mm](M direction){M-stationary in direction $d$};
       \node[format,left of=S,node distance=35mm](QM){$\mathcal{Q}_M$-stationary};
       \node[format,below of=QM,node distance=14mm](Q){$\mathcal{Q}$-stationary};
       \node[format,below of=Q,node distance=10mm](M-stationary){M-stationary};
       \node[format,below of=M-stationary,node distance=10mm](linearized){linearized M-stationary};
       \node[format,right of=linearized,node distance=35mm](AM-stationary){AM-stationary};
\draw[->](S)--(S direction);
\draw[->](S direction)--(Strong M direction);
\draw[->](Strong M direction)--(M direction);
\draw[->](M direction)--(M-stationary);
\draw[->](S)--(QM);
\draw[->](QM)--(Q);
\draw[->](Q)--(QM);
\draw[->](Q)--(M-stationary);
\draw[->](linearized)--(M-stationary);
\draw[->](M-stationary)--(linearized);
\draw[->] (M-stationary)--(AM-stationary);
\end{tikzpicture}

\centering{Fig.2  Relation among stationarities}
\end{figure}
%

\section{Error bound and exact penalty for MPSC}
In this section we show the error bound property under two types of constraint qualifications: one is based on the local decomposition approach and the other is based on the directional quasi-normality.

First we discuss the local decomposition approach. Let $\mathcal{P}(\I_{GH}^*)$ be the set of all (disjoint) bipartitions of $\I_{GH}^*$. For fixed $(\beta_1,\beta_2)\in\mathcal{P}(\I_{GH}^*)$, define
\begin{eqnarray*}
{\rm NLP(\beta_1,\beta_2)}\ \min &&f(z)\\
{{\rm s.t.}}&&g(z)\leq 0,
h(z)=0,G_i(z)=0,\ i\in\I_G^*\cup\beta_1,H_i(z)=0,\ i\in\I_H^*\cup\beta_2.
\end{eqnarray*}
\begin{defi}\rm\label{New CQ definition}
Let $z^*$ be a feasible point of MPSC (\ref{MPSC}). We say that $z^*$  satisfies
\begin{itemize}
\item MPSC piecewise MFCQ/CRCQ/CPLD/RCRCQ/RCPLD/CRSC, if for each $(\beta_1,\beta_2)$ $\in\mathcal{P}(\I_{GH}^*)$, MFCQ/CRCQ/CPLD/RCRCQ/RCPLD/CRSC holds for $({\rm NLP(\beta_1,\beta_2)})$ at $z^*$.
\end{itemize}
\end{defi}

We now compare the piecewise constraint qualifications just defined with  MPSC-MFCQ/-CRCQ/-CPLD as defined in subsection \ref{s2.2}.
{It is easy to see that if  MFCQ/CRCQ/CPLD holds for (TNLP) at $z^*$ then for any $(\beta_1,\beta_2)\in\mathcal{P}(\I_{GH}(z^*))$, MFCQ/CRCQ/CPLD holds for $({\rm NLP(\beta_1,\beta_2)})$ at $z^*$. Hence MPSC-MFCQ/-CRCQ/-CPLD implies MPSC piecewise MFCQ/CRCQ/CPLD.}

{MPSC piecewise MFCQ/CRCQ/CPLD does not imply MPSC-MFCQ/-CRCQ/-CPLD. For example, consider MPSC with constraint system $G(z)=-z_1,H(z)=z_1-z_1^2z_2^2$ at $z^*=(0,0)$. $\nabla G(z)=(-1,0)^T,\nabla H(z)=(1-2z_1z_2^2,-2z_1^2z_2)$. For (TNLP), CPLD does not hold at $z^*$, but for $({\rm NLP(\beta_1,\beta_2)})$, LICQ holds at $z^*$, then MFCQ/CRCQ/CPLD holds at $z^*$.}
 This counter example shows that MPSC piecewise MFCQ/CRCQ/CPLD is strictly weaker than MPSC-MFCQ/-CRCQ/-CPLD.

 Since piecewise constraint qualifications are required to hold for all pieces,  {they} may be harder to verify than the non-piecewise version.
{However sometimes, these two concepts may be equivalent. For example, it was shown in \cite{MengweiYe} that MPSC piecewise RCPLD is equivalent to MPSC-RCPLD.}

In Theorem \ref{CRSCerrorb} we will show that MPSC piecewise CRSC which is the weakest one among all the piecewise constraint qualifications introduced will imply the error bound property.
For this purpose, we first give the following  definition for local error bound property of MPSC (\ref{MPSC}).

\begin{defi}\rm
We say that {\em MPSC local error bound} holds around  $z^*\in\F$   if there exists a neighborhood $V(z^*)$ of $z^*$ and $\alpha>0$ such that
\begin{eqnarray*}
{\rm dist}_\F(z)\leq\alpha\left(\sum_{i=1}^p\max\{g_i(z),0\}
+\sum_{j=1}^q|h_j(z)|+\sum_{i=1}^m\min\{|G_i(z)|,|H_i(z)|\}\right) \ \forall z\in V(z^*).
\end{eqnarray*}
\end{defi}
\begin{thm}\label{CRSCerrorb}
If $z^*\in\F$ verifies MPSC piecewise CRSC, then MPSC local error bound holds in a neighborhood of $z^*$.
\end{thm}
\begin{proof}
Recall that the definition of MPSC piecewise CRSC means that for any $(\beta_1,\beta_2)\in\mathcal{P}(\I_{GH}^*)$,  CRSC holds for nonlinear programs $({\rm NLP(\beta_1,\beta_2)})$ at $z^*$. {When $i\in \I_G^*$, $|H_i(z^*)|>|G_i(z^*)|=0$, there exists a neighborhood $V_G(z^*)$ of $z^*$ such that $|H_i(z)|\geq|G_i(z)|$, then we have $\min\{|G_i(z)|,|H_i(z)|\}=|G_i(z)|$, for $i\in\I_G^*$ and $z\in V_G(z^*)$. Similarly,  there exists a neighborhood $V_H(z^*)$ of $z^*$ such that $\min\{|G_i(z)|,|H_i(z)|\}=|H_i(z)|$, for $i\in\I_H^*$ and $z\in V_H(z^*)$.}  Thus by \cite[Corollary 4.1]{Guo-Zhang-Lin2014} we have  that for $(\beta_1,\beta_2)\in \mathcal{P}(\I_{GH}^*)$, there exist a neighborhood $V_{\beta_1,\beta_2}(z^*)$ and a constant $\alpha_{\beta_1,\beta_2}$ such that
\begin{eqnarray*}
{\rm dist}_\F(z)&\leq&\alpha_{\beta_1,\beta_2} \left (\sum_{i=1}^p\max\{g_i(z),0\}
+\sum_{j=1}^q|h_j(z)|+\sum_{i\in\I_G^*\cup \beta_1}|G_i(z)|+\sum_{i\in\I_H^*\cup \beta_2}|H_i(z)|\right )
\\
&=&\alpha_{\beta_1,\beta_2}\left (\sum_{i=1}^p\max\{g_i(z),0\}
+\sum_{j=1}^q|h_j(z)|+\sum_{i\in\I_G^*}\min\{|G_i(z)|,|H_i(z)|\} \right. \\
&&\qquad\qquad\qquad\left .+\sum_{i\in\I_H^*}\min\{|G_i(z)|,|H_i(z)|\}
+\sum_{i\in\beta_1}|G_i(z)|+\sum_{i\in\beta_2}|H_i(z)|\right ),
\end{eqnarray*}
for all $z\in V_{\beta_1,\beta_2}(z^*)$. Taking $\alpha:=\max_{(\beta_1,\beta_2)\in\mathcal{P}(\I_{GH}^*)}\alpha_{\beta_1,\beta_2}$, $V(z^*):=\cap_{(\beta_1,\beta_2)\in\mathcal{P}(\I_{GH}^*)} V_{\beta_1,\beta_2}(z^*)$, we get for all $z\in V(z^*)$
\begin{eqnarray*}{\rm dist}_\F(z)\leq\alpha\left(\sum_{i=1}^p\max\{g_i(z),0\}
+\sum_{j=1}^q|h_j(z)|+\sum_{i\in\I_G^*}\min\{|G_i(z)|,|H_i(z)|\}\right.\\
\left.+\sum_{i\in\I_H^*}\min\{|G_i(z)|,|H_i(z)|\}
+\sum_{i\in\beta_1}|G_i(z)|+\sum_{i\in\beta_2}|H_i(z)|\right).
\end{eqnarray*}
Finally, it holds for all $(\beta_1,\beta_2)\in \mathcal{P}(\I_{GH}^*)$.
{Set $$\beta_1^*(z):=\{i\in\I^*_{GH}||G_i(z)|=\min\{|G_i(z)|,|H_i(z)|\}\},\quad  \beta_2^*(z):=\I^*_{GH}\setminus\beta_1^*(z),$$ then $(\beta_1^*(z),\beta_2^*(z))\in \mathcal{P}(\I_{GH}^*)$.}

Then we have
\begin{eqnarray*}{\rm dist}_\F(z)\leq\alpha\left(\sum_{i=1}^p\max\{g_i(z),0\}
+\sum_{j=1}^q|h_j(z)|+\sum_{i\in\I_G^*}\min\{|G_i(z)|,|H_i(z)|\}\right.\\
+\sum_{i\in\I_H^*}\min\{|G_i(z)|,|H_i(z)|\}
+\sum_{i\in\beta_1^*(z)}\min\{|G_i(z)|,|H_i(z)|\}\quad\\
\left.+\sum_{i\in\beta_2^*(z)}\min\{|G_i(z)|,|H_i(z)|\}\right)\\
=\alpha\left(\sum_{i=1}^p\max\{g_i(z),0\}
+\sum_{j=1}^q|h_j(z)|+\sum_{i=1}^m\min\{|G_i(z)|,|H_i(z)|\}\right).
\end{eqnarray*}
This completes the proof.
\end{proof}



 Now we discuss the second approach based on the directional quasi-normality. First we need the following calculation.
\begin{lema}\label{lema-distance function}
Under the $l_1$-norm, the distance functions are given by the following expressions for $a,b\in\mathbb{R}:$
\begin{eqnarray*}
&&{\rm dist}_{(-\infty,0]}(a)=\max\{a,0\},\quad {\rm dist}_{\{0\}}(a)=|a|,\\
&&{\rm dist}_{\Omega_{SC}}((a,b))=\min\{|a|,|b|\}=\left(\begin{array}{ll}
a\ or\ b & a=b\geq0,\\
b\quad &|a|>b\geq0,\\
-b\quad &|a|>-b\geq0,\\
a\quad &|b|>a\geq0,\\
-a&|b|>-a\geq0,\\
-a\ or\ -b &a=b\leq0.
\end{array}\right.
\end{eqnarray*}
\end{lema}

\begin{thm}
Let $z^*\in\F$  such that  MPSC directional quasi-normality  holds. Then
 MPSC local error bound holds in a neighborhood of $z^*$.
\end{thm}
\begin{proof}
If MPSC directional quasi-normality holds at $z^*$, then by \cite[Corollary 4.1]{Bai-Ye-Zhang2019}, the set-valued map $F(z):=P(z)-\Lambda$ is metrically subregular at $(z^*,0)$. By the definition of metric subregularity, there exist $\alpha\geq0$ and a neighborhood $N(z^*)$ of $z^*$ such that
\begin{eqnarray*}
{\rm dist}_{F^{-1}(0)}(z)\leq \alpha{\rm dist}_\Lambda(P(z))\quad \forall z\in N(z^*).
\end{eqnarray*}
Recall the distance functions in Lemma \ref{lema-distance function}, we complete the proof.
\end{proof}

By Clarke's exact penalty principle \cite[Proposition 2.4.3]{Clark1990}, we obtain the following exact penalty result immediately.

\begin{thm}
Let $z^*$ be a local optimal solution of MPSC (\ref{MPSC}). If either  MPSC directional quasi-normality or MPSC piecewise CRSC holds at $z^*$, then $z^*$ is a local optimal solution of the penalized problem:
\begin{eqnarray*}
\min f(z)+L_f \alpha\left[\sum_{i=1}^p\max\{0,g_i(z\})+\sum_{j=1}^q|h_j(z)|+\sum_{i=1}^m\min\{|G_i(z)|,|H_i(z)|\}\right],
\end{eqnarray*}
where $\alpha$ is the error bound constant and $L_f$ is the Lipschitz constant of $f$ around $z^*$.
\end{thm}

\section{Conclusions}
\begin{figure}%
		\scriptsize
		 \tikzstyle{format}=[rectangle,draw,thin,fill=white]
		 \tikzstyle{test}=[diamond,aspect=2,draw,thin]
		\tikzstyle{point}=[coordinate,on grid,]
		\begin{tikzpicture}
		\node[format] (MPSC-LICQ){MPSC-LICQ};
		\node[format,below of=MPSC-LICQ,node distance=7mm] (MPSC-MFCQ){MPSC-MFCQ};
		\node[format,below of=MPSC-MFCQ,node distance=7mm] (MPSC-CPLD){MPSC-CPLD};
		\node[format,right of=MPSC-LICQ,node distance=35mm] (LCQ){MPSC Linear CQ};
        \node[format,below of=LCQ,node distance=7mm](MPSC-CRCQ){MPSC-CRCQ};
		\node[format,below of=MPSC-CRCQ,node distance=7mm] (piecewise CRCQ){MPSC piecewise CRCQ};
\node[format,below of=piecewise CRCQ,node distance=7mm](piecewise RCRCQ){MPSC piecewise RCRCQ};
		\node[format,below of=piecewise RCRCQ,node distance=12mm] (piecewise RCPLD){MPSC piecewise RCPLD};
		\node[format,below of=piecewise RCPLD,node distance=13mm](MPSC-RCPLD){MPSC-RCPLD};
		\node[format,left of=MPSC-MFCQ,node distance=30mm](piecewise 
MFCQ){MPSC piecewise MFCQ};
		\node[format,below of=piecewise MFCQ,node distance=7mm](NNAMCQ){MPSC-NNAMCQ};
		\node[format,below of=NNAMCQ,node distance=12mm](pseudo){MPSC  pseudo-normality};
        \node[format,below of=pseudo,node distance=13mm](quasi){MPSC  quasi-normality};
        \node[format,left of=MPSC-LICQ,node distance=65mm](LICQ-d){MPSC-LICQ(d)};
        \node[format,below of=LICQ-d,node distance=13mm](FOSCMS){MPSC-FOSCMS in direction $d$};
        \node[format,below of=LICQ-d,node distance=20mm](SOSCMS){MPSC-SOSCMS in direction $d$};
        \node[format,below of=SOSCMS,node distance=14mm](pseudo directional){MPSC pseudo-normality in direction $d$};
        \node[format,below of=pseudo directional,node distance=12mm](quasi directional){MPSC quasi-normality in direction $d$};
        \node[format,below of=quasi directional,node distance=7mm](M in directional){M-stationarity in direction $d$};
        \node[format,below of=MPSC-CPLD,node distance=32mm](error bound){Metric subregularity/Error bound};
        \node[format,below of=error bound,node distance=7mm](M-stationary){M-stationarity};
        \node[format,below of=MPSC-CPLD,node distance=7mm](piecewise CPLD){MPSC piecewise CPLD};
        \node[format,below of=piecewise CPLD,node distance=12mm](MPSC-CRSC){MPSC piecewise CRSC};
		\node[format,below of=MPSC-RCPLD,node distance=7mm](AM-regular){AM-{regularity}};
		\draw[->] (MPSC-LICQ)--(MPSC-MFCQ);
		\draw[->] (MPSC-LICQ)--(MPSC-CRCQ);
		\draw[->](MPSC-CRCQ)--(piecewise CRCQ);
        \draw[->](piecewise CRCQ)--(piecewise RCRCQ);
		\draw[->](piecewise RCRCQ)--(piecewise RCPLD);
		\draw[->](MPSC-MFCQ)--(MPSC-CPLD);
		\draw[->](MPSC-CPLD)--(piecewise CPLD);
        \draw[->](piecewise CPLD)--(piecewise RCPLD);
		\draw[->](MPSC-CRCQ)--(MPSC-CPLD);
		\draw[->](piecewise RCPLD)--(MPSC-CRSC);
		\draw[->](NNAMCQ)--(piecewise MFCQ);
		\draw[->](piecewise MFCQ)--(NNAMCQ);
		\draw[->](NNAMCQ)--(pseudo);
		\draw[->](pseudo)--(quasi);
		\draw[->](MPSC-MFCQ)--(piecewise MFCQ);
        \draw[->](pseudo)--(pseudo directional);
        \draw[->](quasi)--(quasi directional);
        \draw[->](pseudo directional)--(quasi directional);
		\draw[->](quasi directional)--(error bound);
        \draw[->](quasi directional)--(M in directional);
        \draw[->](M in directional)--(M-stationary);
		\draw[->](MPSC-CRSC)--(error bound);
        \draw[->](error bound)--(M-stationary);
        \draw[->](MPSC-LICQ)--(LICQ-d);
        \draw[->](LICQ-d)--(FOSCMS);
        \draw[->](NNAMCQ)--(FOSCMS);
        \draw[->](FOSCMS)--(SOSCMS);
        \draw[->](SOSCMS)--(pseudo directional);
        \draw[->](piecewise CRCQ)--(piecewise CPLD);
        \draw[->](piecewise RCPLD)--(MPSC-RCPLD);
        \draw[->](MPSC-RCPLD)--(piecewise RCPLD);
        \draw[->](MPSC-RCPLD)--(M-stationary);
        \draw[->](LCQ)--(MPSC-CRCQ);
        \draw[->] (AM-regular)--(M-stationary);
        \draw[->] (MPSC-RCPLD)--(AM-regular);
		\end{tikzpicture}

\centering{Fig.3 Relation among CQs, stationary conditions and error bounds for MPSC}
	\end{figure}
In Figure 3, we give a diagram displaying the relations of various constraint qualifications, stationary conditions and error bounds. Note that in the diagram,  the arrows pointing to stationary points only  hold for  local optimal solutions. MPSC Linear CQ means all defining constraint functions $g,h, G,H$ are all affine. The relation between MPSC piecewise RCPLD and MPSC-RCPLD can be checked easily by using definitions. 
{The proof of relation between MPSC-RCPLD and AM-regularity is similar to \cite[Theorem 4.8]{Andreani etal2019}.}
To obtain all other  relationships, we use definitions and the results presented here together with the results from \cite{Mehlitz MP,Li-Guo,Bai-Ye-Zhang2019,Gfr132}.
From the diagram, we can see that directional conditions in a nonzero critical direction $d$ are weaker than the corresponding nondirectional ones.


%
%
%

\end{document}